\documentclass{amsart}
\usepackage{amsmath,amssymb, amsbsy}
\usepackage[dvips]{graphicx}
\usepackage{color}
\usepackage[latin1]{inputenc}
\usepackage[active]{srcltx}
\usepackage{wrapfig}
\usepackage{graphicx}
\usepackage[usenames,dvipsnames]{pstricks}
\usepackage{epsfig}
 \usepackage{pst-grad} 
 \usepackage{pst-plot} 
\usepackage{epic}
\theoremstyle{plain}

\newtheorem{definition}{Definition}

\newtheorem{proposition}{Proposition}
\newtheorem{remark}{Remark}

\numberwithin{equation}{section} 
\newcommand{\R}{{I\!\!R}}

\newcommand{\p}{\partial}

\newcommand{\re}{{I\!\!R}}
\newcommand{\ren}{\re^N}

\newcommand{\dyle}{\displaystyle}
\newcommand{\ene}{{I\!\!N}}

\newcommand{\io}{\int\limits_\O}

\newcommand{\sob}{W_0^{1,2}}

\renewcommand{\a }{\alpha }
\renewcommand{\b }{\beta }
\renewcommand{\d }{\delta }
\newcommand{\D }{\Delta }
\newcommand{\e }{\varepsilon }

\newcommand{\g }{\gamma}

\renewcommand{\l }{\lambda }

\newcommand{\n }{\nabla }

\newcommand{\s }{\sigma }

\renewcommand{\O }{\Omega }

\newcommand{\inn}{\mbox{ in }}

\newtheorem{Theorem}{Theorem}[section]
\newtheorem{Corollary}[Theorem]{Corollary}
\newtheorem{Definition}[Theorem]{Definition}
\newtheorem{Lemma}[Theorem]{Lemma}
\newtheorem{Proposition}[Theorem]{Proposition}
\newtheorem{remarks}[Theorem]{Remarks}

\begin{document}

\title[KPZ with fractional diffusion]{Towards a deterministic  KPZ equation with fractional diffusion: The stationary problem}
\author[B. Abdellaoui, I. Peral]{Boumediene Abdellaoui and Ireneo Peral }
\thanks{ This work is partially supported by project
MTM2013-40846-P, MINECO, Spain. } 

\address{\hbox{\parbox{5.7in}{\medskip\noindent {B. Abdellaoui, Laboratoire d'Analyse Nonlin\'eaire et Math\'ematiques
Appliqu\'ees. \hfill \break\indent D\'epartement de Math\'ematiques, Universit\'e Abou Bakr Belka\"{\i}d, Tlemcen, \hfill\break\indent Tlemcen 13000,
Algeria.}}}}
\address{\hbox{\parbox{5.7in}{\medskip\noindent{I. Peral, Departamento de Matem\'aticas,\\ Universidad Aut\'onoma de Madrid,\\
        28049, Madrid, Spain. \\[3pt]
        \em{E-mail addresses: }\\{\tt boumediene.abdellaoui@inv.uam.es, \tt ireneo.peral@uam.es}.}}}}

\date{}

\begin{abstract}
In this work we analyze the existence of solution to the fractional quasilinear problem,
\begin{equation*}
\left\{
\begin{array}{rcll}
(-\Delta)^s u &= & |\nabla u|^{q}+\l f & \text{ in }\Omega , \\ u &=& 0 &\hbox{  in } \mathbb{R}^N\setminus\Omega,\\ u&>&0 &\hbox{ in }\Omega,
\end{array}%
\right.
\end{equation*}%
where $\Omega \subset \ren$ is a bounded regular domain ($\mathcal{C}^2$  is sufficient), $s\in (\frac 12, 1)$, $1<q$ and $f$ is a measurable nonnegative function
with  suitable hypotheses.

The analysis is done separately in three cases, subcritical, $1<q<2s$, critical, $q=2s$, and supercritical, $q>2s$.
\end{abstract}

\maketitle

\section{Introduction.}

\label{section00} 

The aim of this paper is to discuss the existence of a weak solution of the following nonlocal elliptic problem, with a gradient term,
\begin{equation}\label{P}
\left\{
\begin{array}{rcll}
(-\Delta)^s u &= & |\nabla u|^{q}+\l f & \text{ in }\Omega , \\ u &=& 0 &\hbox{  in } \mathbb{R}^N\setminus\Omega,\\ u&>&0 &\hbox{ in }\Omega,
\end{array}%
\right.
\end{equation}%
where $\Omega \subset \ren$ is a bounded regular domain ($\mathcal{C}^2$  is sufficient), $s\in (\frac 12, 1)$, $1<q$ and $f$ is a measurable nonnegative
function. Here  $(-\Delta)^s$  means the  classical fractional Laplacian operator of order $2s$ defined by
\begin{equation}\label{fraccionario}
(-\Delta)^{s}u(x):=a_{N,s}\mbox{ P.V. }\int_{\mathbb{R}^{N}}{\frac{u(x)-u(y)}{|x-y|^{N+2s}}\, dy},\, s\in(0,1),
\end{equation}
where
$$a_{N,s}:=2^{2s-1}\pi^{-\frac N2}\frac{\Gamma(\frac{N+2s}{2})}{|\Gamma(-s)|},$$
is the normalization constant  to have the identity
$$(-\Delta)^{s}u=\mathcal{F}^{-1}(|\xi|^{2s}\mathcal{F}u),\, \xi\in\mathbb{R}^{N}, s\in(0,1),$$
in $\mathcal{S}(\mathbb{R}^N)$, the class of Schwartz functions.

 Our goal is to find \textit{natural} conditions on  $f$ in  order to obtain the existence of positive solution.

The problem \eqref{P}  can be seen as a Kardar-Parisi-Zhang stationary equation with fractional diffusion. See \cite{KPZ} for the derivation of the model in the
local case.

In this sense, since the model KPZ assume the growth of the interface in the direction of its normal, it seems to be natural to assume that $s>\frac 12$.

The local problem was widely studied  by many authors, see for instance \cite{BMP0}, \cite{BMP1}, \cite{BMP2}, \cite{BMP3}, \cite{BMP4}, \cite{BGO}, \cite{GMP},
\cite{HMV}, \cite{ADP}, \cite{FM1}, \cite{FM2} and the references therein. In the above works, the case $q\le 2$ (subcritical and critical growth) is deeply
studied. The existence of solution is obtained using \textit{a priori} estimates, that follow using suitable test functions, and comparison principles. In the
last three references some sharp regularity results are obtained. As a consequence, the authors in \cite{ADP} were able to prove a characterization of all
positive solutions which implies a \textit{wild} non uniqueness result. We refer also to the recent paper \cite{hairier} where a fine analysis of the stochastic
case is considered.

The supercritical case, $q>2$, has been studied, for instance, in the references \cite{HMV}, \cite{CON1} and \cite{CON2}. The existence of solution is obtained
using Potential Theory and some fixed point arguments. It is clear that,  in any case, the existence results are guaranteed under regularity hypotheses on $f$ and
smallness condition on $\l$.

We refer to \cite{W}, \cite{W1} for some physical motivations and results for the corresponding evolutionary problem.

The aim of this article is the analysis of the nonlocal case. There are significative differences with the local case. First at all it is necessary to identify
the critical growth in the fractional setting. By homogeneity,  the critical power seems to be $q=2s$ and in fact, this is the threshold to use the comparison
techniques when $q<2s$.  It is worthy to point out that in a such critical growth, there is no known a change of variables similar to the Hopf-Cole change in the
local case. Moreover the techniques of nonlinear test functions, in general, are difficult to adapt in the nonlocal problem.

One of the main tools in our analysis  are sharp estimates on the Green function of the \textit{fractional laplacian } obtained in \cite{BJ}.  Some interesting
results by using such estimates, appear in the papers by Chen-Veron  \cite{CV1} and \cite{CV2}. They consider   nonlinear term $|\n u|^q$ with
$q<p_*:=\frac{N}{N-2s+1}$ and they obtain  sharp existence results for $f$ a Radon measure.

We will concentrate  to cover the range $q\ge p_*$, it seems that  the argument used in \cite{CV1}, \cite{CV2} can not be extended directly to this range and then
we need to use a different approach. To deal with the subcritical case $q<2s$, we will prove a new comparison principles in the spirit of  \cite{APi} and
\cite{P}. This comparison result  allow us to prove the existence and the uniqueness of \textit{nice} solution to problem \eqref{P}. In the case  $q<p_*$, we are
able to prove the  uniqueness of solution for all datum in $L^1$. The techniques based on the  comparison principle has a serious limitation in the critical and
supercritical cases, that is,  for $q\ge 2s$. Such difficulties drive us to only get a super-solution when we start from an ordered family of approximated
problems in order to solve the problem \eqref{P}. To overcame this lack of compactness, we will use estimates from potential theory and we will apply fixed point
argument inspired from the papers  \cite{CON1} and \cite{CON2} in  the local case.

The paper is organized as follows.

In Section \ref{sec2} we begin by some basic results about problem with general  datum in $L^1$ or in the space of Radon measures. As it was observed in the local
case, existence of solution to problem \eqref{P} is strongly related to the regularity of the solution to problem
\begin{equation}
\left\{
\begin{array}{rcll}
(-\Delta)^s v &=& f & \mbox{ in }\Omega, \\ v &=& 0 & \mbox{  in } \mathbb{R}^N\setminus\Omega.
\end{array}%
\right.  \label{PQ}
\end{equation}
In the same  section we will also precise the sense in which the solutions are understood and, for  $s>\frac 12$, we  will establish  the regularity of the
solution $v$ to problem \eqref{PQ} according to the regularity of the datum $f$. In Section \ref{sec3} we state the  comparison principle  to be used in the
subcritical case. The proof relies on a Harnack type inequality of the fractional operator perturbed by a first order linear term. Once obtained the comparison
principle, we are able to prove  existence and uniqueness results for approximated problems and to obtain the uniqueness result for the problem studied in
\cite{CV2}, that is, for $q<p_*$. Problem \eqref{P} with $q<2s$ is treated in Section \ref{sec4}. The proof of existence of solution uses  the comparison
principle and  the construction of  a suitable supersolution. As in the local framework, the existence of a solution will be guaranteed under additional
hypotheses on $f$ and smallness condition on $\lambda$. The compactness argument used in this section can not be used to treat the critical and the supercritical
cases, for this reason the analysis of these two cases will be performed in Section \ref{sec5} by using suitable estimates from Potential Theory and  the
classical Schauder fixed point theorem. Finally, in the last Section we collect some  open problems that seem to be interesting to solve.
\section{Preliminaries and auxiliary results.}\label{sec2}
In this section we present some useful results about the problem
\begin{equation}\label{eq:def}
\begin{cases}
(-\Delta )^s v=\nu &\hbox{   in   }  \Omega,\\ v=0   &\hbox{   in   } \mathbb{R}^N\setminus\Omega,
\end{cases}
\end{equation}
where $\nu$ is a bounded Radon measure. We give   some  definitions about the set of Radon measures and the sense in which a solution to problem \eqref{eq:def} is
considered (see \cite{CV1} and \cite{CV2}).
\begin{definition}\label{def001}
Let $\mu$ be a bounded Radon measure and $\beta>0$. We say that  $\mu\in \mathfrak{M}(\Omega, d^\beta)$  if
$$\int_\Omega d^\beta(x) d|\mu|<+\infty,$$ with $d(x)=\text{dist}(x,\partial\Omega)$.
\end{definition}

In the same way, if $f$ is a locally integrable function, then $f\in \mathfrak{M}(\Omega, d^\beta)$ if
$$\int_\Omega |f(x)| d^\beta(x)dx<+\infty.$$

It is clear that if $f\in L^1(\O)$, then $f\in \mathfrak{M}(\Omega, d^\beta)$ for all $\beta>0$.

Next we  precise the sense in which solutions are defined for this general class of data.

\begin{definition}\label{def1}
We say that $u$ is a weak solution to problem \eqref{eq:def} if $u\in L^1(\O)$, and for all $\phi\in \mathbb{X}_s$, we have
$$
\io u(-\Delta )^s\phi dx =\io \phi d\nu,
$$
where
$$
\mathbb{X}_s\equiv \Big\{\phi\in \mathcal{C}(\ren)\,|\,\text{supp}(\phi)\subset \overline{\O},\,\, (-\Delta )^s\phi(x) \hbox{ is pointwise defined and   }
|(-\Delta )^s\phi(x)|<C \hbox{ in  } \O\Big\}.
$$
\end{definition}

The functional framework to obtain solution to truncated problem is the fractional Sobolev space given in the next definition. See \cite{dine} and \cite{LPPS}.
\begin{Definition}\label{Hs} {For  $0<s<1$,} we define  the fractional Sobolev space of order $s$ as
$$H^s(\mathbb{R}^N):=\{u\in L^2(\mathbb{R}^N)\,\,|\,\,  \int_{\mathbb{R}^N}\int_{\mathbb{R}^N}\frac{|u(x)-u(y)|^2}{|x-y|^{N+2s}}dxdy<+\infty   \}.$$
\end{Definition}
We define now the space $H_0^s (\Omega)$ as the completion of $\mathcal{C}^\infty_0(\Omega)$ with respect to the norm of $H^s(\R^N)$. Notice that if  $u\in
H^s_0(\Omega)$, we have $u=0 \hbox{ a.e. in } \R^N\setminus \Omega$ and we can write
$$
\int_{\ren}\int_{\ren} \frac{|u(x)-u(y)|^2}{|x-y|^{N+2s}}\,dx\,dy=\iint_{D_\Omega} {\frac{|u(x)-u(y)|^2}{|x-y|^{N+2s}}\,dx\,dy}
$$
where
$$
{D_\Omega} := \ren \times \ren \setminus \big( \mathcal{C} \Omega \times  \mathcal{C} \Omega \big) \,.
$$

\begin{definition}\label{tron}
For $\sigma\in\mathbb{R}$, we set
$$T_k(\s) =\max(-k, \min(k,\s))\quad \hbox{    and   } \quad  G_k(\sigma)=\sigma-T_k(\sigma).$$
\end{definition}

From \cite{LPPS}, \cite{CV1} and \cite{AAB} the following result holds.
\begin{Theorem}\label{entropi}
Assume that $f\in L^1(\O)$, then problem \eqref{eq:def}  has a unique weak  solution $u$  that is obtained as the limit of $\{u_n\}_{n\in \mathbb{N}}$, the
sequence of the unique solutions to the approximating problems
\begin{equation}\label{proOO}
\left\{\begin{array}{rcll} (-\Delta)^s u_n &= & f_n(x) & \mbox{  in  }\O,\\ u_n &= & 0 & \mbox{ in } \ren\backslash\O,
\end{array}
\right.
\end{equation}
with $f_n=T_n(f)$.  Moreover,
\begin{equation} \label{tku}
T_k(u_n)\to T_k(u)\hbox{  strongly in }   H^{s}_{0}(\Omega), \quad \forall k > 0,
\end{equation}
\begin{equation} \label{L1u}
u \in L^\theta(\Omega) \,, \qquad  \forall  \ \theta\in \big(1, \frac{N}{N-2s}\big)\,
\end{equation}
 and
\begin{equation}\label{L1du}
\big|(-\Delta)^{\frac{s}{2}}   u\big| \in L^r(\Omega) \,, \qquad \forall  \  r \in \big(1,  \frac{N}{N-s} \big) \,.
\end{equation}
Moreover,
\begin{equation}\label{L1duu}
u_n\to u\mbox{  strongly in } W^{s,q_1}_0(\O)\mbox{  for all   }q_1<\frac{N}{N-2s+1}.
\end{equation}
\end{Theorem}

\

\begin{remarks}

\begin{enumerate}

\

\item In \cite{CV1}, the authors proved the existence and the uniqueness of a weak solution to problem \eqref{eq:def}, in the sense of Definition \ref{def1}
    for all bounded Radon measure $f$.

\item In \cite{AAB}, the more general framework  of the fractional $p$-Laplacian operator with nonnegative datum is studied. The uniqueness of nonnegative
    solution in the entropy setting is proved. Since in this work we are considering the linear case $p=2$, then existence and uniqueness of weak solutions
    are obtained without any sign condition on the datum.

\end{enumerate}
\end{remarks}

Let $\mathcal{G}_s(x,y)$ be the Green function associate the fractional laplacian $(-\Delta)^s$, then $\mathcal{G}_s(x,y)$ solves the problem
\begin{equation}\label{green0}
\left\{\begin{array}{rcll} (-\Delta)^s_y \mathcal{G}_s(x,y) &= & \delta_x & \mbox{  if  }y\in \O,\\ \mathcal{G}_s(x,y) &= & 0 & \mbox{ if } y\in
\ren\backslash\O,
\end{array}
\right.
\end{equation}
where $x\in \O$ is fixed.

In the case where $\O=\ren$, then using the Fourier transform, we know that $\mathcal{G}_{s,{\ren}}(x,y)=\dfrac{C}{|x-y|^{N-2s}}$ for a normalizing constant $C$.
As a consequence we easily get that $ \mathcal{G}_s(x,y)\le \mathcal{G}_{s,{\ren}}(x,y)$.

It is clear that if $s\in (\frac 12,1)$, then the "classical" Gradient of $\mathcal{G}_{s,{\ren}}$ is locally integrable in $\ren$.

Assume now that $\Omega$ is a bounded regular domain of $\ren$, since, in general, we have not an explicit formula to the Green function, then using a
\textbf{"probabilistic tools"}, the authors in \cite{BJ00}, \cite{BJ} and \cite{CZ} (see also \cite{BJ0,BJ1,BJ}) were able to drive some useful properties of the
Green function and its gradient when $s\in (\frac 12, 1)$. These properties are collected in the next Lemma.

{\begin{Lemma}\label{estimmm}
Assume that  $s\in (0,1)$, then
\begin{equation}\label{green00}
\mathcal{G}_s(x,y)\simeq \frac{1}{|x-y|^{N-2s}}\bigg(\frac{d^s(x)}{|x-y|^{s}}\wedge 1\bigg) \bigg(\frac{d^s(y)}{|x-y|^{s}}\wedge 1\bigg),
\end{equation}
in particular we have
\begin{equation}\label{green1}
\mathcal{G}_s(x,y)\le C_1\min\{\frac{1}{|x-y|^{N-2s}}, \frac{d^s(x)}{|x-y|^{N-s}}, \frac{d^s(y)}{|x-y|^{N-s}}\}.
\end{equation}
If $s\in (\frac 12,1)$, then
\begin{equation}\label{green2}
|\n_x \mathcal{G}_s(x,y)|\le C_2 \mathcal{G}_s(x,y)\max\{\frac{1}{|x-y|}, \frac{1}{d(x)}\}.
\end{equation}
Therefore, for all $x,y\in \O$,
\begin{equation}\label{greenm01}
|\n_x \mathcal{G}_s(x,y)|\le \frac{C_3}{|x-y|^{N-2s+1}d^{1-s}(x)}.
\end{equation}
\end{Lemma}}

{
\begin{proof}
Notice that estimates \eqref{green00}, \eqref{green1} follow form \cite{KT2}, \cite{BJ0}, however \eqref{green2} follows from \cite{BJ0}, \cite{BJ1} and \cite{BJ}. Then we have just to prove \eqref{greenm01}. Let $(x,y)\in
\O\times \O$ be such that
$$
\max\{\frac{1}{|x-y|}, \frac{1}{d(x)}\}=\frac{1}{|x-y|},
$$
then using \eqref{green1} and \eqref{green2}, if follows that
\begin{eqnarray*}
|\n_x \mathcal{G}_s(x,y)| &\le & \frac{C_2 \mathcal{G}_s(x,y)}{|x-y|}\le \frac{C_4}{|x-y|^{N-2s+1}}\\ &\le & \frac{C_5}{|x-y|^{N-2s+1}d^{1-s}(x)}.
\end{eqnarray*}
Now, if
$$
\max\{\frac{1}{|x-y|}, \frac{1}{d(x)}\}=\frac{1}{d(x)},
$$
then
\begin{eqnarray*}
|\n_x \mathcal{G}_s(x,y)| &\le & \frac{C_2 \mathcal{G}_s(x,y)}{d(x)}\le \frac{C_4}{d^{1-s}(x)|x-y|^{N-s}}\\ &\le & \frac{C_5}{|x-y|^{N-2s+1}d^{1-s}(x)},
\end{eqnarray*}
where in the last inequality we have used the fact that for all $x,y\in \O$,
$$
\frac{1}{|x-y|^{N-s}}\le \frac{\tilde{C}}{|x-y|^{N-2s+1}},
$$
which follows using the fact that $\O$ is bounded and $s>\frac 12$. Hence we conclude.
\end{proof}

\begin{remarks}\label{obs0}
From \eqref{green00} and \eqref{green1}, we can show that
\begin{equation}\label{neww0}
\mathcal{G}_s(x,y)\le C_1\frac{d^{s}(y)d^{s}(y)}{|x-y|^{N}},
\end{equation}
and for all $\beta\in (0,1)$, we have
\begin{equation}\label{green101}
\mathcal{G}_s(x,y)\le C_1\frac{d^{\beta s}(y)d^{(1-\beta) s}(y)}{|x-y|^{N-s}},
\end{equation}
\begin{equation}\label{green102}
\mathcal{G}_s(x,y)\le C_1\frac{d^{\beta s}(y)}{|x-y|^{N-s(2-\beta)}}.
\end{equation}
\end{remarks}
}

\begin{remarks}\label{obs1}

It is worthy to point out that since the fractional laplacian is the infinitesimal generator of stable Lévy processes, then as it was proved in \cite{BJ00},
\cite{KT}, if $f$ is a regular bounded function, then we have the next representation formula for the unique solution to problem \eqref{eq:def},
\begin{equation}\label{repres1}
u(x)= \io \mathcal{G}_s(x,y)f(y)dy=E^x\bigg(\int_0^{\tau_\O}f(X_\s)d\s\bigg)
\end{equation}
where $(X_\s)$ is $2s$-stable process in $\ren$ and $\tau_\O$ is the first exit time of $X_t$ from $\Omega$. We refer to \cite{BBK}, \cite{APP} for more details
about Levy process and it connection with the fractional Laplacian.
\end{remarks}

We consider now the general case $f\in L^1(\O)$. As it was stated in Theorem \ref{entropi}, the problem \eqref{eq:def} has a unique weak solution $u$ in the sense
of Definition \ref{def1}.

Define
\begin{equation}\label{main-rep}
w(x)=\io \mathcal{G}_s(x,y)f(y)dy,
\end{equation}
we show that $w$ is well defined and that $u=w$.

Let us begin by the next useful tools form Fourier analysis that will be used systematically in the rest of the paper.
\begin{Lemma}\label{ineq1}
Let $0<\lambda<N$,  $1\le p<l<\infty$ be such that  $\dfrac{1}{l}+1=\dfrac{1}{p}+\dfrac{\l}{N}$. For $g\in L^p(\ren)$, we define $$J_\lambda(g)(x)=\int_{\ren}
\dfrac{g(y)}{|x-y|^\l}dy.$$
\begin{itemize}
\item[$a)$] $J_\lambda$ is well defined in the sense that the integral converges absolutely for  almost all $x\in \mathbb{R}^N$. \item[$b)$] If $p>1$, then
    $|| J_\lambda(g)||_l\le c_{p,q}||g||_p$. \item[$c)$] If $p=1$, then $|\{x\in \mathbb{R}^N\,| J_\lambda(g)(x)>\sigma\}|\le \big(\dfrac{
    A||g||_1}{\sigma}\big)^l$.
\end{itemize}
\end{Lemma}
See  for instance Section 1.2 of  Chapter V in \cite{Stein} for the proof.

\

We return now to complete the proof of the fact that $w=u$.

Notice that the function $w$ is well defined: this follows since $f\in L^1(\O)$, taking into consideration that
$$
\mathcal{G}_s(x,y)\le \frac{C}{|x-y|^{N-2s}},
$$
and using Lemma \ref{ineq1}.

From the definition of $\mathcal{G}_s$, we deduce that $w=0$ in $\ren\backslash \O$.

Recall that $u$ is obtained as a limit of the sequence $\{u_n\}_n$ where $u_n$ is the unique solution to the approximating problem \eqref{proOO}. Choosing $f_n$
regular and bounded, then by Remark \ref{obs1}, the solution $u_n$ is given by
\begin{equation}\label{rep1n}
u_n(x)=\io \mathcal{G}_s(x,y)f_n(y)dy.
\end{equation}
Thus
$$
|w(x)-u_n(x)|\le \io \mathcal{G}_s(x,y)|f(y)-f_n(y)|dy\le C\io \dfrac{|f(y)-f_n(y)|}{|x-y|^{N-2s}}dy.
$$
Hence by the classical lemma \ref{ineq1}, it holds that
$$
|\{x\in \mathbb{R}^N\,|w(x)-u_n(x)|>\sigma\}|\le \big(\dfrac{ A||f-f_n||_{L^1(\O)}}{\sigma}\big)^l, \mbox{  for all }l<\frac{N}{N-2s}
$$
Letting $n\to \infty$, we conclude that $u_n\to w$ in the Marcinkiewicz space $M^{\frac{N}{N-2s}}(\O)$ and then $u_n\to w$ strongly in the Lebesgue space
$L^\theta(\O)$ for all $\theta<\frac{N}{N-2s}$. Hence $w=u$ and then $u$, the unique solution to problem \eqref{eq:def}, is also given by
\begin{equation}\label{rep100}
u(x)=\io \mathcal{G}_s(x,y)f(y)dy.
\end{equation}

As a consequence of the   estimates in Lemma \ref{estimmm}, the authors in \cite{CV2} obtain the following regularity result.
\begin{Theorem}\label{key}
Suppose that $s\in (\frac 12, 1)$ and let $f\in \mathfrak{M}(\O)$, a Radon measure. Then the problem \eqref{eq:def} has a unique weak solution $u$ in the sense of
Definition \ref{def1}  such that,
\begin{enumerate}
\item  $|\n u|\in M^{p_*}(\O)$, the Marcinkiewicz space,  with $p_*=\frac{N}{N-2s+1}$ and as a consequence
 $u\in W^{1,\theta}_0(\Omega)$ for all $\theta<p_*$. Moreover
\begin{equation}\label{dd}
||u||_{W^{1,\theta}_0(\Omega)}\le C(N,q,\O)||f||_{\mathfrak{M}(\O)}.
\end{equation}
\item For $f\in L^1(\O)$, setting\; $T: L^1(\O)\to  W^{1,\theta}_0(\Omega)$, with $T(f)=u$, then $T$ is a compact operator.
\end{enumerate}
\end{Theorem}
\begin{remark}\label{rm01}
If $f\in L^1(\Omega,d^\beta)$ with $0\le \beta\le s$, then for $p\in(1, \dfrac{N}{N+\beta-2s})$ there exists $c_p>0$ such that
$$||u||_{W^{2s-\gamma,p}}\le ||f||_{L^1(\Omega,d^\beta)},$$
where $\gamma=\beta+\frac N{p'}$  if   $\beta>0$ and $\gamma>\frac N{p'}$ if $\beta=0$.
\end{remark}

We will use Theorems \ref{entropi} and \ref{key} as a starting point of our analysis.

To end this part, we recall  the classical Hardy-Littlewood-Sobolev inequality.
\begin{Lemma}\label{ineq2}
Let $0<\lambda<N$, $\theta,\, \gamma>1$ with  $\dfrac{1}{\theta}+\dfrac{1}{\gamma}+\dfrac{\l}{N}=2$. Assume  that $g\in L^\theta(\ren)$ and $k\in L^\gamma(\ren)$,
then
$$
\Big|\iint_{\ren\times \ren}\frac{g(x)k(y)}{|x-y|^\l}dxdy\Big|\le C(N,\l,\theta)||k||_{L^\gamma(\ren)}||g||_{L^\theta}.
$$
\end{Lemma}
For the proof we refer, for instance, to  Section 4.3 in \cite{Lieb-Loss} or to the paper \cite{Lieb}.

\subsection{A technical result}
In what follows we will assume that $s\in (\frac 12,1)$.

In the local case, $s=1$, if $u$ is the solution to the corresponding problem \eqref{eq:def}, it is known that $T_k(u)\in \sob(\O)$ for all $k>0$. This follows
using $T_k(u)$ as a test function in \eqref{eq:def}. To prove a similar result in the fractional setting in the next subsection we need the following  result.
\begin{Lemma}\label{singular}
Let $\Omega$ be a bounded smooth domain. Assume $\dfrac 12<s<1$ and $\alpha\in \re$ verifying  $1<\a<2s$. Then the problem
\begin{equation}\label{rho}
\left\{
\begin{array}{rcll}
(-\Delta)^s \rho &= & \dfrac{1}{d^\a(x)} & \text{ in }\Omega , \\ \rho &=& 0 &\hbox{  in } \mathbb{R}^N\setminus\Omega,
\end{array}%
\right.
\end{equation}
has a positive solution $\rho$ such that $\rho\in L^\infty(\O)$ and $\rho^{\frac{\beta+1}{2}}\in  H^s_0(\Omega)$ for all $\beta>\max\{\frac{\a}{2s-\a}, 1\}$. In
particular $\rho^{\theta}\in  H^s_0(\Omega)$ for all $\theta>\max\{\frac{s}{2s-\a}, 1\}$.
\end{Lemma}
\begin{proof} Recall that $\O$ is a regular domain, since $s>\dfrac 12$, the following Hardy type inequality was obtained in \cite{Loss},
\begin{equation}\label{hardy}
C\io \dfrac{\phi^2}{d^{2s}(x)}dx\le ||\phi||^2_{H^s_0(\O)} \mbox{ for all }\phi\in H^s_0(\O),
\end{equation}
where $C>0$ depend only on $s$, $N$ and $\Omega$.

Define $\rho_n$ to be the unique solution to the approximating problem
\begin{equation}\label{rhon}
\left\{
\begin{array}{rcll}
(-\Delta)^s \rho_n &= & \dfrac{1}{d^\a(x)+\frac 1n} & \text{ in }\Omega , \\ \rho_n &=& 0 &\hbox{  in } \mathbb{R}^N\setminus\Omega.
\end{array}%
\right.
\end{equation}
It is clear that $\rho_n\in L^\infty(\O)\cap H^s_0(\O)$, $\rho_n>0$ in $\O$ and $\rho_n\le \rho_{n+1}$ for all $n$.

Since $\a<2s$, we can pick-up $\beta>1$ such that $\a<\dfrac{2s\beta}{\beta+1}$.

Taking $\rho^\beta_n$ as a test function in \eqref{rhon} and by using the numerical inequality,
\begin{equation}\label{alge3}
(a-b)(a^{\g}-b^{\g})\ge c|a^{\frac{1+\g}{2}}-b^{\frac{1+\g}{2}}|^2 \mbox{  for all }a, b \in \re^+\mbox{   and }\g, c >0,
\end{equation}
(see \cite{Kas1}), it follows that
\begin{equation}\label{energy}
C||\rho^{\frac{\beta+1}{2}}_n||^2_{H^s_0(\O)}\le \io \dfrac{\rho^\beta_n}{d^\a(x)}dx.
\end{equation}
Thus using the H\"older inequality, it holds
$$
\io \dfrac{\rho^\beta_n}{d^\a(x)}dx\le \Big(\io \dfrac{\rho^{\beta+1}_n}{d^{2s}(x)}dx\Big)^{\frac{\b}{\beta+1}}\Big(\io \dfrac{1}{d^{(\beta+1)(\a-\frac{2s
\beta}{\beta+1})}(x)}dx\Big)^{\frac{1}{\beta+1}}.$$ Notice that, since $\a<\dfrac{2s\beta}{\beta+1}$, then $(\beta+1)(\a-\dfrac{2s \beta}{\beta+1})<0$,
$$\io \dfrac{1}{d^{(\beta+1)(\a-\frac{2s \beta}{\beta+1})}(x)}dx<\infty.$$

Therefore, by Hardy inequality \eqref{hardy},
\begin{equation}\label{HH}
\io \dfrac{\rho^\beta_n}{d^\a(x)}dx\le C\Big(\io \dfrac{\rho^{\beta+1}_n}{d^{2s}(x)}dx\Big)^{\frac{\b}{\beta+1}}\le
C||\rho^{\frac{\beta+1}{2}}_n||^{\frac{\beta}{\beta+1}}_{H^s_0(\O)}.
\end{equation}
Hence, from \eqref{energy} and \eqref{HH}, we have $||\rho^{\frac{\beta+1}{2}}_n||_{H^s_0(\O)}\le C$ for all $n\in \mathbb{N}$. Then,  there exists $\rho$ such
that $\rho^{\frac{\beta+1}{2}}\in H^s_0(\O)$, $\rho_n\uparrow \rho$ strongly in $L^{\frac{2^*_s(\beta+1)}{2}}(\O)$ and $\rho^{\frac{\beta+1}{2}}_n\rightharpoonup
\rho^{\frac{\beta+1}{2}}$ weakly in $H^s_0(\O)$.

It is not difficult to show that $\rho$ is a solution to problem \eqref{rho} in the sense of distribution and that $\dfrac{\rho^{\beta+1}_n}{d^\a(x)}\to
\dfrac{\rho^{\beta+1}}{d^\a(x)}$ strongly in $L^1(\O)$. In the same way we can show that $\rho\in L^\s(\O)$ for all $\s>0$.

To prove that $\rho\in L^\infty(\O)$, we  follows \cite{LPPS} where  the classical Stampacchia argument  in \cite{St} is adapted to the fractional setting. For
the reader convenience we give some details, mainly the estimates involving the Hardy inequality .

Take $G^\beta_{k}(\rho_n)$, with $k>0$, as a test function in \eqref{rhon}. Notice that by using \eqref{alge3}, we have
$$
\begin{array}{c}
(\rho_n(x)-\rho_n(y))\bigg(G^\beta_k (\rho_n(x))- G^\beta_k(\rho_n(y))\bigg)\ge C \bigg|G^{\frac{\beta+1}{2}}_k (\rho_n(x))-
G^{\frac{\beta+1}{2}}_k(\rho_n(y))\bigg|^{2},
\end{array}
$$
then,
$$
\begin{array}{lll}
C \dyle \int\int_{D_\O}\frac{|G^{\frac{\beta+1}{2}}_{k}(\rho_n(x))-G^{\frac{\b+1}{2}}_{k}(\rho_n(y))|^2}{|x-y|^{N+2s}}\, dxdy & \leq & \dyle
\int_{\Omega}\dfrac{G^\beta_{k}(\rho_n(x))}{d^\a(x)}dx \\ & \le & \dyle  \Big(\io \dfrac{G^{\beta+1}_k(\rho_n(x))}{d^{2s}(x)}dx\Big)^{\frac{\b}{\beta+1}}
\Big(\int_{A_k}d^{(\beta+1)(\frac{2s\beta}{\beta+1}-\a)}(x)dx\Big)^{\frac{1}{\beta+1}}.
\end{array}
$$
Letting $n\to \infty$  it follows that
$$
\int\int_{D_\O}\frac{|G^{\frac{\beta+1}{2}}_{k}(\rho(x))-G^{\frac{\b+1}{2}}_{k}(\rho(y))|^2}{|x-y|^{N+2s}}\, dxdy\leq \Big(\io
\dfrac{G^{\beta+1}_k(\rho(x))}{d^{2s}(x)}dx\Big)^{\frac{\b}{\beta+1}} |A_{k}|^{\frac{1}{\beta+1}},
$$
where $A_{k} =\{ x \in \Omega\,: \, |u(x)|  \geq k\}$. Hence by Hardy inequality,
$$
||G^{\frac{\beta+1}{2}}_k(\rho)||^{\frac{1}{\beta+1}}_{H^s_0(\O)}\le C|A_k|^{\frac{1}{\beta+1}},
$$
and by using the Sobolev Inequality, there results that
$$
S|| G^{\frac{\beta+1}{2}}_k(\rho)||_{L^{2^*_s} (\Omega)}\leq  C |A_k|.
$$

Let $h>k$, since $A_h\subset A_k$,  we obtain that
$$
(h-k)^{\frac{\beta+1}{2}} |A_h|^{\frac{1}{2^*_s}}\leq C|A_k|,
$$
that is,
$$
|A_h|\leq \frac{C |A_k|^{2^*_s} }{(h-k)^{\frac{2^*_s(\beta+1)}{2}}}\,.
$$
Since $2^*_s>1$, by the classical numerical lemma of Stampacchia in \cite{St}, there exists $k_0>0$ such that $|A_h|=0$ for all $h\ge k_0$. Thus $\rho$ is
bounded.
\end{proof}
\subsection{Some regularity results.}
We start  by proving  the following regularity result of $T_k(v)$ that will be a useful tool in the subsequent arguments.
\begin{Theorem}\label{key01}
Assume that $f\in L^1(\Omega)$ and define $v$ to be the unique weak solution to problem \eqref{eq:def}, then $T_k(v)\in W^{1,\a}_0(\Omega)\cap H^s_0(\Omega)$  for
any $1<\a<2s$, moreover $$ \io |\n T_k(v)|^{\a}\, dx\le Ck^{\a-1}||f||_{L^1(\O)}.
$$
\end{Theorem}
\begin{proof} Without loss of generality, we can assume that $f\ge 0$.  By Theorem \ref{entropi}, $T_k(v)\in H^s_0(\Omega)$.
Let $\mathcal{G}_s$ be the Green kernel of $(-\Delta)^s$, then
\begin{equation}\label{repre}
v(x)=\io \mathcal{G}_s(x,y) f(y)dy.
\end{equation}
Thus
$$
|\n T_k(v)|=|\n v(x)|\chi_{\{|u(x)|<k\}} \le\Big(\io |\n_x \mathcal{G}_s(x,y)|f(y)dy\Big)\chi_{\{|v(x)|<k\}}.
$$
Fix $1<\a<2s$, then using H\"older inequality, it follows that
\begin{eqnarray*}
|\n T_k(v)|^{\a} &\le & \Big(\io |\n_x \mathcal{G}_s(x,y)|f(y)dy\Big)^{\a}\chi_{\{|v(x)|<k\}} \\ & \le & \dyle \Big(\io \dfrac{|\n_x
\mathcal{G}_s(x,y)|}{\mathcal{G}_s(x,y)} \mathcal{G}_s(x,y)f(y)dy\Big)^{\a}\chi_{\{|v(x)|<k\}} \\& \le & \dyle \Big(\io \bigg[\dfrac{|\n_x
\mathcal{G}_s(x,y)|}{\mathcal{G}_s(x,y)}\bigg]^{\a} \mathcal{G}_s(x,y)f(y)dy\Big)\Big(\io \mathcal{G}_s(x,y)f(y)dy\Big)^{\a-1}\chi_{\{|v(x)|<k\}} \\ &\le & \dyle
\Big(\io \bigg[\dfrac{|\n_x \mathcal{G}_s(x,y)|}{\mathcal{G}_s(x,y)}\bigg]^{\a} \mathcal{G}_s(x,y)f(y)dy\Big)v^{\a-1}(x)\chi_{\{|v(x)|<k\}}\\ &\le & \dyle
k^{\a-1}\Big(\io \bigg[\dfrac{|\n_x \mathcal{G}_s(x,y)|}{\mathcal{G}_s(x,y)}\bigg]^{\a} \mathcal{G}_s(x,y)f(y)dy\Big)\chi_{\{|v(x)|<k\}}.
\end{eqnarray*}
From Lemma \ref{estimmm}, we know that
$$
\mathcal{G}_s(x,y)\le C_1\min\bigg\{\frac{1}{|x-y|^{N-2s}}, \frac{d^s(x)}{|x-y|^{N-s}}, \frac{d^s(y)}{|x-y|^{N-s}}\bigg\},
$$
and
$$
|\n_x \mathcal{G}_s(x,y)|\le C_2 \mathcal{G}_s(x,y)\max\bigg\{\frac{1}{|x-y|}, \frac{1}{d(x)}\bigg\}.
$$
Thus, setting $h(x,y)=\dfrac{|\n_x \mathcal{G}_s(x,y)|}{\mathcal{G}_s(x,y)}$, we reach that
\begin{eqnarray*}
\io |\n T_k(v)|^{\alpha}dx &\le & \dyle k^{\alpha-1}\io f(y) \Big(\io h(x,y)^\a\mathcal{G}_s(x,y) dx\Big)dy.
\end{eqnarray*}
It is clear that $h(x,y)^\alpha\le C\Big(\frac{1}{|x-y|^{\a}}+\frac{1}{d^{\a}(x)}\Big)$.
 Therefore, we conclude that
\begin{eqnarray*}
\io |\n T_k(v)|^{\a}dx &\le & \dyle C k^{\frac{\a}{\a'}}\Big\{\io f(y)\bigg(\io\frac{1}{|x-y|^{\a}}\mathcal{G}_s(x,y) dx\bigg)dy +\io f(y)\bigg(\io
\frac{1}{d^{\a}(x)}\mathcal{G}_s(x,y) dx\bigg)dy\Big\}\\ &\le & C k^{\alpha-1}(I_1+I_2).
\end{eqnarray*}
Respect to $I_1$, using the fact that $\mathcal{G}_s(x,y)\le \frac{C}{|x-y|^{N-2s}}$, it holds
$$
I_1\le C\io f(y)\io \bigg(\frac{1}{|x-y|^{\a}}\mathcal{G}_s(x,y) dx\bigg)dy\le C\io f(y)\bigg(\io\frac{1}{|x-y|^{N-2s+\a}}dx\bigg)dy.
$$
Since $\a<2s$, then, for $R>>2\text{diam}\,(\O)$,
$$\io\frac{1}{|x-y|^{N-2s+\a}}dx\le \int_{B_R(y)}\frac{1}{|x-y|^{N-2s+\a}}dx\le CR^{2s-\a}. $$
We deal now with $I_2$. Consider $\rho$, the unique solution to problem
\begin{equation*}\label{rho1}
\left\{
\begin{array}{rcll}
(-\Delta)^s \rho &= & \dfrac{1}{d^\a(x)} & \text{ in }\Omega , \\ \rho &=& 0 &\hbox{  in } \mathbb{R}^N\setminus\Omega,
\end{array}%
\right.
\end{equation*}
obtained in Lemma \ref{singular}. Since $\a<2s$,  by Lemma \ref{singular}, $\rho\in L^\infty(\O)$ .  Thus
$$
I_2\le C\io f(y)\bigg(\io \frac{1}{d^\a}\mathcal{G}_s(x,y) dx\bigg)dy=C\io f(y)\rho(y)dy\le C||f||_{L^1(\O)}.
$$
Combining the above estimates, it follows that
$$
\io |\n T_k(v)|^{\a}dx\le \dyle C k^{\a-1}||f||_{L^1(\O)}.
$$
Hence we conclude. \end{proof}

We next precise some results in order to find  the regularity of $v$ when $f$ is assumed to be more regular.

{
\begin{Lemma}\label{key2-corr}
Suppose that $f\in L^m(\O)$ with $m\ge 1$ and define $v$ to be the unique solution to problem
\begin{equation}\label{gener1-corr}
\left\{
\begin{array}{rcll}
(-\Delta)^s v &= & f & \mbox{ in }\Omega , \\ v &=& 0 &\hbox{  in } \mathbb{R}^N\setminus\Omega,
\end{array}%
\right.
\end{equation}
with $s>\frac12$. Then for all $p<\frac{m N}{N-m(2s-1)}$, there exists a positive constant $C\equiv C(\O, N, s,p)$ such that
\begin{equation}\label{dd11-corr}
\bigg\||\n v| d^{1-s}\bigg\|_{L^p(\Omega)}\le C||f||_{L^m(\O)}.
\end{equation}
Moreover,
\begin{enumerate}
\item If $m=\frac{N}{2s-1}$, then $|\n v|d^{1-s}\in L^p(\Omega)$ for all $p<\infty$. \item If $m>\frac{N}{2s-1}$, then $v\in \mathcal{C}^{1, \sigma}(\Omega)$
    for some $\sigma\in (0,1)$. and
$$
\bigg\| |\n v|d^{1-s}\bigg\|_{L^\infty(\O)}\le C||f||_{L^m(\O)}.
$$
\end{enumerate}
\end{Lemma}}

{
\begin{proof}
Without loss of generality we can assume that $f\ge 0$. Since $m\ge 1$, then we know that $|\n v|\in M^{p_*}(\O)$ with $p_*=\frac{N}{N-2s+1}$ and
$$
||\n v||_{L^{\theta}(\Omega)}\le C(N,q,\O)||f||_{L^1(\O)},
$$
for all $\theta<p_*$. It is clear that
$$
v(x)=\io \mathcal{G}_s(x,y)f(y)dy.
$$
Thus
$$
\n v(x)=\io \n_x \mathcal{G}_s(x,y)f(y)dy,
$$
then by \eqref{greenm01}, it holds that
$$
|\n v(x)|\le \io |\n_x \mathcal{G}_s(x,y)|f(y)dy\le \frac{C}{d^{1-s}(x)}\io \frac{f(y)}{|x-y|^{N-2s+1}}dy.
$$
Define
$$
g(x)=\io \frac{\tilde{f}(y)}{|x-y|^{N-2s+1}}dy
$$
where
\begin{equation}\label{ff0cor}
\tilde{f}(x)= \left\{
\begin{array}{rcll}
f(x) & \mbox{  if } & x\in \Omega , \\  0 &\mbox{  if } &  x\in \mathbb{R}^N\setminus\Omega.
\end{array} \right.
\end{equation}
Hence
$$
|\n v|d^{1-s}\le g\mbox{  in   }\O.
$$
Since $\tilde{f}\in L^m(\ren)$, then by Lemma \ref{ineq1}, $g\in L^{\frac{m N}{N-m(2s-1)}}(\ren)$ if $m<\frac{N}{2s-1}$ and in particular $g\in L^p(\O)$ for all
$p\le \frac{m N}{N-m(2s-1)}$. Hence we conclude.

If $m=\frac{N}{2s-1}$, then repeating the above argument with $m_1<m$, it holds that $g\in L^p(\O)$ for all $p<\infty$.

Assume that $m>\frac{N}{2s-1}$, then using H\"older inequality,
$$
\io \frac{f(y)}{|x-y|^{N-2s+1}}dy\le C(\O)||f||_{L^m(\O)} \mbox{  for all }x\in \O.
$$
Thus
$$
\bigg\| |\n v|d^{1-s}\bigg\|_{L^\infty(\O)}\le C||f||_{L^m(\O)}.
$$
Recall that
$$
\n v(x)=\io \n_x \mathcal{G}_s(x,y)f(y)dy,
$$
then by the dominated convergence Theorem and Vital's lemma and since $m>\frac{N}{2s-1}$, then $\n v$ is continuous in $\O$.

To shaw that $|\n v|$ is H\"older continuous we will prove that $|\n v|\in W^{2s-\g-1, l}(\Omega)$ where $0<\gamma<2s-1$ and $(2s-\gamma-1)l>N$. We proceed with
similar duality arguments as in \cite{CV1}. Let  $\phi\in \mathcal{C}^\infty_0(\Omega)$, then
$$
|\io \phi (-\Delta)^s v dx|=|\io f\phi dx|\le ||f||_{L^m(\O)}||\phi||_{L^{m'}(\O)}.
$$
Setting $l=\dfrac{Nm}{N-\gamma m}$ with $\g\in [0,1]$ is such that $\gamma<\min\{(2s-1), \frac{N}{m}\}$, then $m'=\dfrac{l'N}{N-l'\gamma}$. It is clear that
$\frac{N}{2s-1}<m<l$.

Now, using the definition of $q$ and by Sobolev inequality, it follows that
$$
|\io \phi (-\Delta)^s v dx|\le ||f||_{L^m(\O)}||\phi||_{L^{\frac{l'N}{N-l'\gamma}}(\O)}\le C||f||_{L^m(\O)}||\phi||_{W^{\gamma,l'}_0(\Omega)}
$$
Thus $(-\Delta)^s v\in W^{-\gamma, l}(\Omega)$ and
$$
||(-\Delta)^s v||_{W^{-\gamma, l}(\Omega)}\le C||f||_{L^m(\O)}.
$$
Recalling that $(-\Delta)^s$ realize an isomorphism between $W^{2s-\gamma, l}(\Omega)$ and $W^{-\gamma, l}(\Omega)$, hence we conclude that $v\in W^{2s-\gamma,
l}(\Omega)$ and
$$
||v||_{W^{2s-\gamma, l}(\Omega)}\le C||f||_{L^m}.
$$
Since $2s-\gamma>1$, then $|\n v|\in W^{2s-\gamma-1, l}(\Omega)$. Now, using the definition of $l$, we reach easily that $(2s-\gamma-1)l>N$. Thus by the
fractional Morrey inequality,   $|\n v|\in \mathcal{C}^{0, \sigma}(\Omega)$ with $\sigma=\frac{N}{(2s-\g-1)l}$
\end{proof}
}

{
\begin{remarks}\label{loc-reg}
Notice that, form Lemma \ref{estimmm}, we get
$$
v(x)\le C\, d^s(x)\io \frac{f(y)}{|x-y|^{N-s}}dy.
$$
Therefore, by Lemma \ref{ineq1}, we obtain that $\dfrac{v}{d^s}\in L^{\frac{mN}{N-ms}}(\O)$ and
$$
||\frac{v}{d^s}||_{L^{\frac{mN}{N-ms}}(\O)}\le C(\O,s,N,m)||f||_{L^m(\O)}.
$$
Since
$$
|\n (v d^{1-s})|=|d^{1-s}\n v+(1-s)\frac{v}{d^s}\n d|\le d^{1-s}|\n v|+(1-s)\frac{v}{d^s}|\n d|,
$$
then using the fact that $|\n d|=1$\, a.e in $\O$, then for all $p\le \frac{mN}{N-m(2s-1)}$, we get
$$
|\n (v d^{1-s})|^p\le C_1 d^{(1-s)p}|\n v|^p+C_2(\frac{v}{d^s})^p.
$$
Hence
$$
\io |\n (v d^{1-s})|^p dx \le C_1\io d^{(1-s)p}|\n v|^p dx +C_2\io (\frac{v}{d^s})^p dx.
$$
It is clear that $\frac{mN}{N-m(2s-1)}<\frac{mN}{N-ms}$, thus using H\"older inequality, it follows that
$$
\bigg(\io |\n (v d^{1-s})|^p dx\bigg)^{\frac{1}{p}}:=||v d^{1-s}||_{W^{1,p}_0(\O)}\le C||f||_{L^m(\O)}.
$$
\end{remarks}
}

{
In the case where $f\in L^1(\O)\cap L^m_{loc}(\O)$ where $m>1$, then the above regularity results hold locally in $\O$. More precisely we have
\begin{Lemma}\label{key2-locc}
Assume that $f\in L^1(\O)\cap L^m_{loc}(\O)$ with $m>1$. Let $v$ the unique solution to problem \eqref{gener1-corr}, then for any $\O_1\subset\subset\O$ and for
all $p\le \frac{m N}{N-m(2s-1)}$, there exists $C:=C(\O,\O_1,p)$ such that
\begin{equation}\label{dd11loc}
||\n v||_{L^p(\Omega_1)}\le C(||f||_{L^1(\O)}+||f||_{L^m(\O_1)}).
\end{equation}
Moreover,
\begin{enumerate}
\item If $m=\frac{N}{2s-1}$, then $|\n v|\in L^p_{loc}(\Omega)$ for all $p<\infty$. \item If $m>\frac{N}{2s-1}$, then $v\in \mathcal{C}^{1, \sigma}(\Omega)$
    for some $\sigma\in (0,1)$.
\end{enumerate}
\end{Lemma}
\begin{proof}
As in the proof of the previous Lemma, and since $f\in L^1(\O)$, then $|\n v|\in M^{p_*}(\O)$ with $p_*=\frac{N}{N-2s+1}$,
$$
||\n v||_{L^{\theta}(\Omega)}\le C(N,q,\O)||f||_{L^1(\O)},
$$
for all $\theta<p_*$. Since $\O_1\subset\subset\O$, then $\text{dist}(\O_1,\p\O)\ge C_1>0$. Consider the set
$$
\O_2=\{y\in \O: \text{dist}(y,\p\O)\le \frac{C_1}{2}\},
$$
then for $x\in \O_1$, we have
$$
\n v(x)=\io \n_x \mathcal{G}_s(x,y)f(y)dy.
$$
Thus, using the fact that  $\text{dist}(x,y)\ge \frac{C_1}{2}$ for all $x\in \O_1$ and $y\in \O_2$, it holds that
\begin{eqnarray*}
|\n v(x)| &\le & \int_{\O_2} |\n_x \mathcal{G}_s(x,y)|f(y)dy + \int_{\O\backslash \O_2} |\n_x \mathcal{G}_s(x,y)|f(y)dy \\ &\le & \frac{C}{d^{1-s}(x)}\int_{\O_2}
\frac{f(y)}{|x-y|^{N-2s+1}}dy + \int_{\O\backslash \O_2} |\n_x \mathcal{G}_s(x,y)|f(y)dy\\ &\le & C||f||_{L^1(\O)}+ \int_{K} |\n_x \mathcal{G}_s(x,y)|f(y)dy\\
&\le & C||f||_{L^1(\O)}+ \int_{K} \frac{f(y)}{|x-y|^{N-2s+1}}dy\\
\end{eqnarray*}
where $K$ is a compact set of $\O$ such that $\O\backslash \O_2\subset\subset K$. Since $f\in L^m(K)$, then using Lemma
\ref{ineq1}, we get the desired estimates.
\end{proof}
}
{
\begin{remarks}
The result of Lemma \ref{key2-locc} will be useful in order to get $\mathcal{C}^1$ regularity using bootstrapping argument if we have global bounds in $L^1$ and local family of bounds in a suitable $L^m_{loc}$.
\end{remarks}
}
\section{Comparison principle and applications}\label{sec3}
In this section we will prove a comparison principle that extend the one proved in \cite{APi} in the local case. More precisely, we will prove the following
result.
\begin{Theorem}\label{compa2}(Comparison Principle).
Let $g\in L^1(\O)$ be a nonnegative function. Assume that for all $\xi_1,\xi_2\in \mathbb{R}^N$,
$$H:\Omega\times \mathbb{R}^N\rightarrow\mathbb{R}^+ \hbox{  verifies   } |H(x,\xi_1)-H(x,\xi_2)|\le C b(x)|\xi_1-\xi_2|$$
where $b\in L^\s(\O)$ for some $\s>\frac{N}{2s-1}$. Consider  $w_1, w_2$ two positive functions such that $w_1, w_2\in W^{1,p}(\O)$ for all $p<p_*$, $(-\Delta)^s
w_1, (-\Delta)^s  w_2\in L^1(\O)$,  $w_1\le w_2$ in $\ren \setminus\Omega$ and
\begin{equation}\label{eq:compar1} \left\{
\begin{array}{rcll}
(-\Delta)^s  w_1 &\le &  H(x,\nabla w_1)+g & \hbox{ in  }\Omega,\\ (-\Delta)^s w_2 & \ge & H(x,\nabla w_2)+g & \hbox{ in  }\Omega.
\end{array} \right.
\end{equation}
Then, $w_2\ge w_1$ in $\Omega$.
\end{Theorem}
In order to prove Theorem \ref{compa2} we will follow  the arguments used by Porretta in \cite{P} for differential equations and for that  we need some results on
problems with first order term.

Let us begin by recalling the following Harnack inequality proved in \cite{BJ}.
\begin{Proposition}\label{H}(Harnack inequality).
Assume that $B\in (L^\sigma(\O))^N$ with $\s>\frac{N}{2s-1}$ and let $w\in \mathcal{C}^{1, \alpha}(\Omega)$ be a nonnegative function in $\ren$ such that
$$(-\Delta)^s w-\langle B(x),\n w\rangle =0 \mbox{ in } B_R$$
with $B_R\subset\subset \Omega$. Then
$$
\sup_{B_R}w\le C\inf_{B_R} w
$$
where $C\equiv C(\O, B_R)$.
\end{Proposition}
\subsection{A uniqueness result for a related problem}
 We prove the following uniqueness result.
\begin{Lemma}\label{key3} Let $B$ be  a vector field in $\Omega$.
Assume that $B\in (L^\sigma(\O))^N$ with $\s>\frac{N}{2s-1}$ and let $w$ be a solution to the problem
\begin{equation}\label{grad1}
\left\{
\begin{array}{rcll}
(-\Delta)^s w &= &\langle B(x),\n w\rangle & \mbox{ in }\Omega , \\  w &=& 0 &\hbox{  in } \mathbb{R}^N\setminus\Omega,
\end{array}%
\right.
\end{equation}
with $|\n w| \in  M^{p_*}(\O)$, the Marcinkiewick space, then $w=0$.
\end{Lemma}
\begin{proof}
 Sine $\s \frac{N}{2s-1}$ then using H\"older inequality it holds that $|\langle B(x),\n w\rangle|\in L^1(\O)$.

We claim that $w\in \mathcal{C}^{ 1, \alpha}(\Omega)$ for some $\a\in (0,1)$.

We divide the proof of the claim into two steps:

{
\textsc{First step.}  Suppose that $|\n w|\in L^{m}_{loc}(\Omega)$ with $m=\frac{N \sigma}{(2s-1)\sigma-N}$, then setting $h(x)=|\langle B(x),\n w\rangle|$ and
taking into account the regularity of $B$,  it follows that $h\in L^{\frac{N}{2s-1}}_{loc}(\Omega)\cap L^1(\O)$. Going back to \eqref{grad1} and by the first
point in Lemma \ref{key2-locc}, we conclude that $|\nabla w|\in L^p_{loc}(\Omega)$ for all $p<\infty$. Thus $h\in L^\sigma_{loc}(\O)\cap L^1(\O)$. Since
$\s>\frac{N}{2s-1}$, by the second point in Lemma \ref{key2-locc}, we conclude that $|\n w|\in \mathcal{C}^{0, \alpha}(\Omega)$ for some $\a\in (0,1)$ and the
result follows.}

{
\textsc{Second step.} We prove the  regularity result of the first step, namely that $|\n w|\in L^{\frac{N \sigma}{(2s-1)\sigma-N}}_{loc}(\Omega)$.\\
We will use a bootstrapping argument.\\
Since $|\n w| \in M^{p_*}(\O)$, then $h\in L^{l_1}(\Omega)$ for $1<l_1<\dfrac{p_*\sigma}{\sigma+p_*}$. Fix $l_1$ as above, then using Lemma \ref{key2-locc}, it
follows that $|\n w|\in L^{r_1}_{loc}(\Omega)$ with $r_1=\dfrac{l_1N}{N-l_1(2s-1)}$. Hence by the H\"{o}lder inequality we conclude that $h\in
L^{l_2}_{loc}(\Omega)$ with $l_2=\dfrac{r_1\sigma}{\sigma+r_2}$. Using again Lemma \ref{key2-locc} , it follows that $|\n w|\in L^{r_2}_{loc}(\Omega)$ with
$r_2=\dfrac{l_2N}{N-l_2(2s-1)}$.}

It is clear that $r_2>r_1$. Define by iteration the  sequence $\{r_n\}_n$ by $$r_{n+1}=\dfrac{N\sigma r_n}{N\sigma-r_n(\sigma(2s-1)-N)}.$$ If for some $n_0$,
$r_{n_0}\ge \dfrac{\sigma N}{(2s-1)\sigma-N}$, then the result follows.

We argue by contradiction. Assume that $r_n<\frac{\sigma N}{(2s-1)\sigma-N}$ for all $n$. It is easy  to show that $\{r_n\}_n$ is an increasing sequence. Hence
there exists $\bar{r}$ such that $r_n\uparrow \bar{r}\le \frac{\sigma N}{(2s-1)\sigma-N}$. Thus $\bar{r}=\frac{N\sigma \bar{r}}{N\sigma-\bar{r}(\sigma(2s-1)-N)}$,
hence $\bar{r}=0$, a contradiction with the fact that $\{r_n\}_n$ is an increasing sequence.

Therefore there exists $n_0\in \ene$ such that $r_{n_0}\ge \frac{\sigma N}{(2s-1)\sigma-N}$ and the claim follows.

Let  us prove now that $w\le 0$. If  by contradiction, $C=\max_{x\in \O} w(x)>0$, then there exists $x_0\in \O$ such that $w(x_0)=C$. We set $w_1=C-w$, then
$w_1\ge 0$ in $\ren$, $w_1(x_0)=0$ and
$$
(-\Delta)^s w_1-B(x)\n w_1=0 \mbox{  in   }\O.
$$
Consider $B_R=B_r(x_0)\subset\subset \O$. By  applying the Harnack inequality in Proposition \ref{H}  to $w_1$, we conclude that $\sup_{B_r(x_0)}w_1\le C(\O,
B_r)\inf_{B_r(x_0)}w_1=0$. Thus $w_1\equiv 0$ in $B_r(x_0)$. Since $\O$ is a bounded domain, then, applying Harnack inequality  a finite number of steps, we prove
that $w_1=0$ in $\O$. Thus $C\le 0$ and then $w\le 0$.

The linearity of the problem permits to apply similar arguments to  $-w$ (that is also a solution to \eqref{grad1}). Thus $w\equiv 0$ and the result follows.
\end{proof}

\subsection{Existence for an auxiliary problem}
Let us prove now the following existence result for an auxiliary problem.
\begin{Lemma}\label{key4}
Assume that $B\in (L^{\sigma_1}(\O))^N$ with $\s_1>\frac{N}{2s-1}$ and $0\le f\in L^{\s_2}(\O)$ with $\s_2>\frac{N}{2s}$, then the problem
\begin{equation}\label{grad2}
\left\{
\begin{array}{rcll}
(-\Delta)^s w &= & \langle B(x),\n w\rangle +f& \mbox{ in }\Omega , \\  w &=& 0 &\hbox{  in } \mathbb{R}^N\setminus\Omega,
\end{array}%
\right.
\end{equation}
has a unique nonnegative bounded solution $w$ such that {$w\in W^{1,p}_0(\O)$ for all $p<p_*$ and $|\n w|\in L^a_{loc}(\O)$ } for all $a<\infty$. In addition, if $\s_2>\frac{N}{2s-1}$, then $\nabla
w\in (\mathcal{C}^{0, \alpha}(\Omega))^N$ for some $\a\in (0,1)$.
\end{Lemma}
\begin{proof} It is clear that the regularity of the solution follows using the same iteration argument as in the proof of the regularity result in Lemma
\ref{key3}. Let us prove  the existence part.

Fix $p<p_*$ be such that $p'<\sigma_1$ and define the operator $T: W^{1, p}_0(\O)\to W^{1, p}_0(\O)$ by setting $w=T(u)$ where $w$ solves
\begin{equation}\label{grad3}
\left\{
\begin{array}{rcll}
(-\Delta)^s w &= &\langle B(x),\n u_+\rangle +f& \mbox{ in }\Omega , \\  w &=& 0 &\hbox{  in } \mathbb{R}^N\setminus\Omega.
\end{array}%
\right.
\end{equation}
Since $u\in W^{1, p}_0(\O)$, then the existence and the uniqueness of $w$ follows by using approximating argument and the results of \cite{CV2}. It is clear that
if $w$ is a fixed point of $T$, then $w$ is a nonnegative solution to \eqref{grad2}. To show that $T$ has a fixed point, we will use the Schauder fixed point
theorem  (  see Theorem 11.3 in \cite{GT}).

From the result of \cite{CV2}, we conclude that $T$ is a compact operator.

We claim that there exists $M>0$ such that it $u=\l T(u)$ with $\l\in [0,1]$, then $||u||_{W^{1, p}_0(\O)}\le M$.

To prove the claim we argue by contradiction.

Assume that there exists sequences $\{\l_n\}\subset [0,1]$ and $\{u_n\}_n$ such that $u_n=\l_n T(u_n)$ and $||u||_{W^{1, p}_0(\O)}\to \infty$ as $n\to \infty$.
Define $v_n=\dfrac{u_n}{||u_n||_{W^{1, p}_0(\O)}}$, then $||v_n||_{W^{1, p}_0(\O)}=1$ and $v_n$ solves
\begin{equation}\label{grad21}
\left\{
\begin{array}{rcll}
(-\Delta)^s v_n &= &\langle B(x),\n (v_n)_+ \rangle+\l_n\dfrac{f}{||u_n||_{W^{1, p}_0(\O)}}& \mbox{ in }\Omega , \\  v_n &=& 0 &\hbox{  in }
\mathbb{R}^N\setminus\Omega.
\end{array}%
\right.
\end{equation}
It is clear that $v_n\ge 0$ and $||\langle B(x),\n (v_n)_+ \rangle||_{L^1(\O)}\le C$,  thus for all $l<p_*$, there exists a positive constant $C$ depending only
on the data and independents of the sequence $\{v_n\}_n$ such that
$$
||v_n||_{W^{1, l}_0(\O)}< C,
$$
and
$$
||\n v_n||_{W^{2s-1-\g, l}_0(\O)}\le C \mbox{  where }\g=\frac{N}{l}.
$$
Up to a subsequence, we find  that $v_n\rightharpoonup v$ weakly in $W^{1, l}_0(\O)\mbox{  for all  }l<p_*$, $v\in W^{1, l}_0(\O)$, $v\ge 0$ and $v$ solves
\begin{equation*}
\left\{
\begin{array}{rcll}
(-\Delta)^s v &= &\langle B(x),\n v\rangle & \mbox{ in }\Omega , \\  v &=& 0 &\hbox{  in } \mathbb{R}^N\setminus\Omega.
\end{array}%
\right.
\end{equation*}
From Lemma \ref{key3} we obtain that $v=0$ and from the compactness result of \cite{CV2} we obtain  that $v_n\to v$ strongly in $W^{1, l}_0(\O)\mbox{  for all
}l<p_*$, in particular for $l=p$ and then $||v||_{W^{1, p}_0(\O)}=1$ which is a contradiction. Hence the claim follows.

Thus $T$ has a fixed point and then problem  \eqref{grad2} has a nonnegative solution.

The uniqueness immediately follows. Indeed, if $w_1$ and $w_2$ are two solutions to problem \eqref{grad2}, then $\tilde{w}=w_1-w_2$ solves
\begin{equation*}
\left\{
\begin{array}{rcll}
(-\Delta)^s \tilde{w} &= & \langle B(x), \n \tilde{w}\rangle  & \mbox{ in }\Omega , \\  \tilde{w} &=& 0 &\hbox{  in } \mathbb{R}^N\setminus\Omega.
\end{array}%
\right.
\end{equation*}
By Lemma \ref{key3}, we know that $\tilde{w}=0$, thus $w_1=w_2$.
\end{proof}

\begin{remark}
We are able to prove the existence result in Lemma   \ref{key4} without the positivity condition on $f$, in fact, consider $w_1$ and $w_2$ the solutions to
problem \eqref{grad2} with data $f_+$ and $f_-$ respectively. Setting $w=w_1-w_2$, then $w$ solve
$$
(-\Delta)^s w=\langle B(x),\n w \rangle+f \mbox{ in }\Omega,
$$
with the same regularity.
\end{remark}
A remarkable result derives from the following  observations.

 Since $s>\frac 12$, we set $E=W^{2s, 2}(\O)\cap H^s_0(\O)$. Then if $w\in E$, it holds that $|\n w|\in W^{2s-1, 2}(\O)$, hence $|\n w|\in
 L^{\frac{2N}{N-2(2s-1)}}(\Omega)$.  By hypothesis $B\in (L^{\sigma_1}(\O))^N$ with $\s_1>\frac{N}{2s-1}$,  then $|\langle B(x), \n w\rangle|\in L^2(\Omega)$ for
 all $w\in E$.

Define $L(w)=(-\Delta)^s w-\langle B(x), \n w\rangle  $; then
$$L:L^2(\O)\to L^{2}(\Omega), \hbox{    with }  Dom(L)=E.$$
By Lemma \ref{key3} we have that $Ker(L)=\{0\}$. Thus $L^{-1}: L^2(\Omega)\to E$ is a well defined compact operator. Therefore by the Fredholm alternative theorem
we conclude that for all $f\in L^2(\Omega)$, the problem \eqref{grad3} has a unique solution $u\in E$.

Now, observe that for $u,v\in E$, we have
\begin{eqnarray*}
\langle L(u),v\rangle & = & \io (-\Delta)^s u\, v dx -\io \langle B(x), \n u\rangle  v\,dx\\ &=& \io u(-\Delta)^s v dx -\io \text{div}(B(x) v) u dx\\ &=& \io
u\bigg((-\Delta)^s v-\text{div}(B(x) v)\bigg)u dx.
\end{eqnarray*}
Hence, by defining
$$
K(v)=(-\Delta)^s v-\text{div}(B(x)v),
$$
then we have,
$$
\langle K(v),u\rangle=\langle v,L u\rangle,
$$
that is $K$ is the adjoint operator of  $L$. Since $0=\text{dim Ker}(L)=\text{dim Ker}(K)$, then for all $f\in L^2(\Omega)$ the problem
\begin{equation}\label{grad02}
\left\{
\begin{array}{rcll}
(-\Delta)^s v-\text{div}(B(x)v) &= & f & \mbox{ in }\Omega , \\  v &=& 0 &\hbox{  in } \mathbb{R}^N\setminus\Omega,
\end{array}%
\right.
\end{equation}
has a unique solution $u\in E$. In particular, we have:

{
\begin{Corollary}\label{cor11}
For all $f\in L^{\s_2}(\O)$ with $\s_2>\frac{N}{2s}$, there exists a unique solution $v$ to problem \eqref{grad02} such that $v\in H^s_0(\O)$ with $|\n v|\in L^{\frac{2N}{N-2(2s-1)}}(\O)$. In addition, if $\s_2>\frac{N}{2s-1}$, then $v\in \mathcal{C}^{0, \alpha}(\Omega)\cap L^\infty(\O)$ for some $\a\in (0,1)$.
If $f\ge 0$, then $v\ge 0$.
\end{Corollary}}

\subsection{Proof of Theorem \ref{compa2}}
We are now able to prove the comparison principle in Theorem  \ref{compa2}.

Consider $w=w_1-w_2$, then $w \in W^{1,p}(\O)$ for all $p<p_*$ and  $(-\Delta)^s  w\in L^1(\O)$.

We have just to show that $w^+=0$. It is clear that $w\le 0$ in $\mathbb{R}^N\setminus\Omega$. By \eqref{eq:compar1}, it follows that
$$
\begin{array}{lll}
(-\Delta)^s w &\le & H(x,\nabla w_1)-H(x,\nabla w_2)\le b(x)|\n w|.
\end{array}
$$
Now, using Kato's inequality (see for instance \cite{LPPS}) we get
\begin{equation}\label{ineq}
(-\Delta)^s w_{+}\le b(x)|\nabla w_{+}|, \:\: w_{+}=\rm max\{w,0\}\in W^{1,q}_0(\Omega)\mbox{  for  all  } q<p_*.
\end{equation}
Let $v$ be the unique positive bounded solution to problem
\begin{equation}\label{grad002}
\left\{
\begin{array}{rcll}
(-\Delta)^s v+\text{div}(\mathcal{F}(x)v) &= & 1 & \mbox{ in }\Omega , \\  v &=& 0 &\hbox{  in } \mathbb{R}^N\setminus\Omega,
\end{array}%
\right.
\end{equation}
where
$$
\mathcal{F}(x)= \left\{
\begin{array}{rcll}
b(x)\dfrac{\n w_+(x)}{|\n w_+(x)|}  & \mbox{ if }|\n w_+(x)|\neq 0\\ 0 & \mbox{  if } |\n w_+(x)|=0.
\end{array}%
\right.
$$

Recall that $b\in L^\s(\O)$ for some $\s>\frac{N}{2s-1}$, then $\mathcal{F}\in L^\s(\O)$. Now, using Corollary \ref{cor11} with $f\equiv 1$, it holds that $v\in L^\infty(\Omega)\cap \mathcal{C}^{1, \alpha}(\Omega)$ for some $\a\in (0,1)$.

Using an approximation argument we can use $v$ as a test function in \eqref{ineq} to obtain that
$$
\dyle \io (-\Delta)^s w_+ vdx\le \int_\Omega b(x) |\nabla w_+(x)|v(x)dx.
$$
On the other hand we have
$$
\begin{array}{lll}
\dyle \io (-\Delta)^s w_+ vdx & = & \dyle \io w_+(-\Delta)^s v dx\\ & =& \dyle\io w_+(-\text{div} \mathcal{F}(x) v)dx+\io w_+ dx\\ &=&\dyle \int_\Omega b(x)
|\nabla w_+(x)|v(x)dx + \io w_+dx.
\end{array}
$$
Hence $\dyle\io w_+\le 0$ and then $w\le 0$ in $\Omega$. Thus we conclude. \hskip 3cm $\square$

As a  byproduct of the previous result we obtain the following uniqueness results.
\begin{Corollary}\label{aplic1}
Let $g\in L^1(\O)$ be a nonnegative function. Suppose that $q\ge 1$ and $a>0$, then the problem
\begin{equation}\label{aprox} \left\{
\begin{array}{rcll}
(-\Delta)^s w &= & \dfrac{|\n w|^q}{a+|\n w|^q} +g & \text{   in }\Omega , \\  w &=& 0 &\hbox{  in } \mathbb{R}^N\setminus\Omega,
\end{array} \right.
\end{equation}
has a unique nonnegative solution $w$ such that $w\in W^{1, \theta}_0(\O)\mbox{  for all  }\theta<p_*$ and $T_k(w)\in H^s_0(\O)\cap W^{1, \alpha}(\Omega)$ for all
$\a<2s$.
\end{Corollary}
\begin{Corollary}
Consider the problem
\begin{equation}\label{aprox2} \left\{
\begin{array}{rcll}
(-\Delta)^s w &= & |\n w|^q +\l g & \text{  in }\Omega , \\  w &=& 0 & \hbox{  in } \mathbb{R}^N\setminus\Omega,
\end{array} \right.
\end{equation}
with $1<q<p_*$ and $g\in L^1(\Omega)$, $g\ge 0$. Then there exists $\lambda^*$ such if   $\l<\l^*$, problem \eqref{aprox2} has a unique positive solution $w$ such
that $w\in W^{1, \theta}_0(\O)\mbox{  for all  }\theta<p_*$ and $T_k(w)\in H^s_0(\O)$ for all $k>0$.
\end{Corollary}
\begin{proof} The existence and regularity can be seen in \cite{CV2}. We prove the uniqueness.
Indeed if $w_1$ and $w_2$ are two positive solutions to problem \eqref{aprox2} with the above regularity, defining $\bar{w}=w_1-w_2$, then $(-\Delta)^s \bar{w}\in
L^1(\O)$, $\bar{w}\in W^{1, q}_0(\O)\mbox{  for all  }q<p_*$, $T_k(\bar{w})\in H^s_0(\O)$ for all $k>0$ and
$$
(-\Delta)^s \bar{w}\le H(x,\nabla w_1)-H(x,\nabla w_2)\le b(x)|\n \bar{w}|
$$
where $b(x)=q\Big(|\n w_1|+|\n w_2|\Big)^{q-1}$. Since $q<p_*$, then $q'>p_*'$, hence $b\in L^\s(\O)$ for some $\s>\frac{N}{2s-1}$. Thus using the comparison
principle in Lemma \ref{compa2}, we conclude that $\bar{w}_+=0$. In the same way and setting $\widehat{w}=w_2-w_1$, we obtain that $\widehat{w}_+=0$. Thus
$w_1=w_2$.
\end{proof}
 In particular we can state the following comparison principle.
\begin{Theorem}\label{compag}
Assume that $g\in L^1(\O)$ is a nonnegative function. Let $w_1, w_2$ be two nonnegative functions such that $w_1, w_2\in W^{1,p}(\O)$ for some $1\le q<p_*$,
$(-\Delta)^s  w_1, (-\Delta)^s  w_2\in L^1(\O)$,  $w_1\le w_2$ in $\ren \setminus\Omega$ and
\begin{equation}\label{eq:compar100U} \left\{
\begin{array}{rcll}
(-\Delta)^s  w_1 &\le &  |\n w_1|^q+g & \hbox{ in  }\Omega,\\ (-\Delta)^s w_2 & \ge & |\n w_2|^q+g & \hbox{ in  }\Omega.
\end{array} \right.
\end{equation}
Then, $w_2\ge w_1$ in $\Omega$.

As a consequence, if $1\le q<p_*$ and $g\equiv 0$, then the unique weak solution to problem \eqref{aprox2} is $w=0$.
\end{Theorem}
\section{The subcritical problem: Existence results via comparison arguments. }\label{sec4}
In this section we consider the problem
\begin{equation}\label{qqA}
\left\{
\begin{array}{rcll}
(-\Delta)^s u &= & |\nabla u|^{q}+\l f & \mbox{ in }\Omega , \\ u &=& 0 &\hbox{  in } \mathbb{R}^N\setminus\Omega,\\ u&>&0 &\hbox{ in }\Omega,
\end{array}%
\right.
\end{equation}%
where $s\in (\dfrac 12,1)$, $q<2s$ and $f\in L^\sigma(\O)$ for some convenient $\s>1$. The main goal of this section is to show that, under additional hypotheses
on $f$, we are able to build a suitable supersolution and then the comparison principle in Theorem \ref{compa2} allows us to use a monotony argument in order to
prove the existence of a minimal positive solution.

\begin{remark}\label{nonn}
Notice that in the local case, the existence of a solution is guaranteed under the condition
\begin{equation}\label{CC}
\inf_{\phi\in \mathcal{C}^\infty_0(\O)}\dfrac{\dyle\io |\n \phi|^{q'}dx}{\dyle\io f|\phi|^{q'}dx}>0.
\end{equation}
By using the spectral theory, it is clear that the above condition holds if $f\in L^\s(\O)$ for some $\s>\frac{N}{q'}$.
\end{remark}

\subsection{ A radial supersolution}
We will start by building  a radial  supersolution with the required regularity.

Define $w(x)=(1-|x|^\alpha)_+$ where $1<\alpha<2s$. Since  $w$ is a radial function, then following closely the computations in \cite{FV}, it holds that
\begin{eqnarray*}
(-\Delta )^s w(x) & = & r^{-2s}\int_{1}^\infty\bigg((w(r)-w(\s r))+(w(r)-w(\frac{r}{\s}))\s^{2s-N}\bigg)\s(\s^2-1)^{-1-2s}H(\sigma)d\sigma
\end{eqnarray*}
where $r=|x|$ and $H$ is a continuous positive function defined in $[1, \infty)$ with $H(\s)\backsimeq \s^{2s}$ as $\s\to \infty$. Fix $r_0<1$ to be chosen later,
then for $r<r_0$, we have
\begin{eqnarray*}
(-\Delta )^s w(x) & = &
r^{-2s}\int_{1}^{\frac{1}{r}}\bigg(r^\alpha(\s^\alpha-1)+r^\alpha(\frac{1}{\s^\alpha}-1)\s^{2s-N}\bigg)\s(\s^2-1)^{-1-2s}H(\sigma)d\sigma\\ &+& r^{-2s}
\int_{\frac{1}{r}}^\infty\bigg((1-r^\alpha)+r^\alpha(\frac{1}{\s^\alpha}-1)\s^{2s-N}\bigg)\s(\s^2-1)^{-1-2s}H(\sigma)d\sigma\\ &=&
r^{\alpha-2s}\int_{1}^{\frac{1}{r}}(\s^\alpha-1)(1-\s^{2s-N-\alpha})\s(\s^2-1)^{-1-2s}H(\sigma)d\sigma\\ &+& r^{\alpha-2s}
\int_{\frac{1}{r}}^\infty\bigg((\frac{1}{r^\alpha}-1)-(1-\frac{1}{\s^\alpha})\s^{2s-N}\bigg)\s(\s^2-1)^{-1-2s}H(\sigma)d\sigma\\ &=& r^{\alpha-2s} F(r)
\end{eqnarray*}
with
\begin{eqnarray*}
F(r)& = & \int_{1}^{\frac{1}{r}}(\s^\alpha-1)(1-\s^{2s-N-\alpha})\s(\s^2-1)^{-1-2s}H(\sigma)d\sigma\\ & + &
\int_{\frac{1}{r}}^\infty\bigg((\frac{1}{r^\alpha}-1)-(1-\frac{1}{\s^\alpha})\s^{2s-N}\bigg)\s(\s^2-1)^{-1-2s}H(\sigma)d\sigma.\\
\end{eqnarray*}
We claim that $F(r)\ge C(r_0)>0$ for all $r\in (0,r_0)$. By a direct computation we find that $F'(r)<0$, hence to conclude, we have just to show that $F(r_0)\ge
C(r_0)$ for suitable $r_0<1$.

Notice that $(\s^\alpha-1)(1-\s^{2s-N-\alpha})\s(\s^2-1)^{-1-2s}H(\sigma)>0$ for all $\s>1$. On the other hand we have $(1-\frac{1}{\s^\alpha})\s^{2s-N}\le
(1-\frac{1}{\s_0^\alpha})\s_0^{2s-N}$ where $\s_0=(\frac{N+\a-2s}{N-2s})^{\frac{1}{\alpha}}$. Thus for $\s>\frac{1}{r}$ we have
$$
(\frac{1}{r^\alpha}-1)-(1-\frac{1}{\s^\alpha})\s^{2s-N}\ge (\frac{1}{r^\alpha}-1)-(1-\frac{1}{\s_0^\alpha})\s_0^{2s-N}.
$$
Defining $r_0\equiv \bigg(\frac{N+\alpha-2s}{N+\alpha-2s+\s_0^{2s-N}}\bigg)^{\frac{1}{\a}}<1$, for $r\le r_0$ and $\s\ge \frac{1}{r}$, it holds that
$$
(\frac{1}{r^\alpha}-1)-(1-\frac{1}{\s^\alpha})\s^{2s-N}\ge 0.
$$
Combining the above estimate we reach that $F(r)\ge C(r_0)>0$ for all $r\le r_0$. Notice that $|\n w|=\alpha |x|^{\alpha-1}\chi_{\{|x|<1\}}$.  Since
$1<\alpha<2s$, setting $w_1=Cw$ for some $C>0$, we obtain that $w_1$ satisfies
\begin{equation}\label{super1}
\left\{
\begin{array}{rcll}
(-\Delta)^s w_1 &\ge & |\nabla w_1|^{q}+\l \dfrac{1}{|x|^{2s-\alpha}} & \mbox{ in }B_{r}(0), \\ w_1 &\ge & 0 &\hbox{  in } \mathbb{R}^N\setminus B_r(0),\\ w_1
&>&0 &\hbox{ in }B_r(0),
\end{array}%
\right.
\end{equation}%
for some $\l>0$.

It is clear that, modulo a rescaling argument, the above construction holds in any bounded domain.

In the case   $p_*<q<2s$, we can guess a positive supersolution in the form
$$S(x)= A |x|^{-\alpha},\quad A,\quad 0<\alpha< N-2s.$$
By direct calculation  we obtain
$$(-\Delta )^s S(x)= C_{N,s}(\alpha) A |x|^{-\alpha-2s}, \quad C_{N,s}(\alpha)>0,$$
and $|\nabla S(x)|\le A\alpha |x|^{-(\alpha+1)}.$

Therefore to have a radial solution in the whole $\mathbb{R}^N$, the following identity must be verified,
$$C_{N,s}(\alpha)|x|^{-\alpha-2s}\ge \alpha^q A^{q-1}|x|^{-(\alpha+1)} \hbox{ for all } x\in \mathbb{R}^N,$$
that is, necessarily $\alpha=\dfrac{2s-q}{q-1}$ and the condition $q<2s$ appears in a natural way.

Hence, it is sufficient to pick-up $A$ such that $\alpha^q A^{q-1}\le C_{N,s}(\alpha)$.

\noindent If the source term $f\in L^\infty(\Omega)$, then we just have to choose $\lambda>0$ small enough in order to have a supersolution  $S$ in $\Omega$.

Since $q>\frac{N}{N-2s+1}\equiv p_*$, then $S\in W^{1,q}_{loc}(\mathbb{R}^N)$ and any translation $\tilde{S}(x)= S(x-x_0)$ is also a supersolution, thus choosing
$x_0\in \ren\backslash \overline{\O}$, $\tilde{S}(x)$ is a bounded supersolution.
\subsection{A first result on the existence of weak solution}
We have the next existence result.
\begin{Theorem}\label{exit1}
Assume $f\in L^\infty(\Omega)$.  Let $w$ be a bounded supersolution to \eqref{qqA} such that $w\in W^{1,q}_0(\O)\cap L^\infty(\O)$. Then problem \eqref{qqA} has a
solution $u$ such that $w\in W^{1,\a}_0(\O)$ for all $\a<2s$.
\end{Theorem}
\begin{proof} Let $u_n$ be the unique solution to the approximating problem
\begin{equation}\label{aprox1} \left\{
\begin{array}{rcll}
(-\Delta)^s u_n &= & \dfrac{|\n u_n|^q}{1+\frac 1n |\n u_n|^q} +\l f & \text{   in }\Omega , \\  u_n &=& 0 &\hbox{  in }\mathbb{R}^N\setminus\Omega.
\end{array} \right.
\end{equation}
By the comparison principle in Theorem \ref{compa2}, it follows that $u_n\le u_{n+1}\le w$ for all $n$. Hence,  there exists $u$ such that $u_n\uparrow u$
strongly in $L^{q^*}(\O)$. Let $g_n(x)=\dfrac{|\n u_n|^q}{1+\frac 1n |\n u_n|^q} +\l f$ and define $\rho$ to be the unique solution to the problem
\begin{equation}\label{inter} \left\{
\begin{array}{rcll}
(-\Delta)^s \rho &= & 1 &\text{   in }\Omega , \\  \rho &=& 0 &\hbox{  in } \mathbb{R}^N\setminus\Omega.
\end{array} \right.
\end{equation}
Using $\rho$ as a test function in \eqref{aprox1} and since $u_n\le w$, it follows that
$$\io g_n(x)\rho\le C  \hbox{ for all }  n.$$

We claim that the sequence $\{u_n\}_n$ is bounded in $W^{1,a}_0(\O)$ for all $a<2s$. We follow closely the same ideas as in  the proof of Theorem \ref{key01}. We
have that
$$
u_n(x)=\io \mathcal{G}_s(x,y) g_n(y)dy.
$$
Hence
$$
|\n u_n(x)|\le\io |\n_x \mathcal{G}_s(x,y)|g_n(y)dy.
$$
Fix $1<\a<2s$ and define $h(x,y)=\max\bigg\{\dfrac{1}{|x-y|},\dfrac{1}{d(x)}\bigg\}$, it holds
\begin{eqnarray*}
|\n u_n(x)|^{\a} &\le & \Big(\io |\n_x \mathcal{G}_s(x,y)|g_n(y)dy\Big)^{\a}\le \Big(\io h(x,y) \mathcal{G}_s(x,y)g_n(y)dy\Big)^{\a}\\& \le & \dyle \Big(\io
(h(x,y))^{\a} \mathcal{G}_s(x,y)g_n(y)dy\Big)\Big(\io \mathcal{G}_s(x,y)g_n(y)dy\Big)^{\a-1}\\ &\le & \dyle \Big(\io (h^{\a}(x,y)
\mathcal{G}_s(x,y)g_n(y)dy\Big)u^{\a-1}_n(x)\\ &\le & \dyle \io \Big(h^{\a}(x,y) \mathcal{G}_s(x,y)g_n(y)dy\Big)w^{\a-1}(x)\\ &\le & \dyle \io \Big(h^{\a}(x,y)
\mathcal{G}_s(x,y)|\n u_n(y)|^qdy\Big)w^{\a-1}(x)+ \l \io \Big(h^{\a}(x,y) \mathcal{G}_s(x,y)f(y)dy\Big)w^{\a-1}(x).
\end{eqnarray*}
Thus
\begin{eqnarray*}
\dyle\io |\n u_n|^{\a}dx & \le & \dyle \io|\n u_n(y)|^q \Big(\io h^{\a}(x,y) \mathcal{G}_s(x,y)w^{\a-1}(x)\Big)dy\\ & + & \dyle \l \io f(y)\Big(\io
h^{\a}(x,y)\mathcal{G}_s(x,y)w^{\a-1}(x)dx\big)dy\\ &\equiv & J_1+J_2.
\end{eqnarray*}
Since $w\in L^\infty(\O)$ and $\a<2s$,  following the same computations as in the proof of Theorem \ref{key01}, we reach that
$$
J_1\le C\io|\n u_n(y)|^q \Big(\io h^{\a}(x,y) \mathcal{G}_s(x,y)dx\Big)dy\le C\io|\n u_n(y)|^q dy
$$
and
$$
J_2\le C\io f(y)dy.
$$
Therefore we conclude that
$$
\io|\n u_n(x)|^\a dx\le C_1\io|\n u_n(x)|^qdx +C_2.
$$
Choosing $\a>q$ and by H\"older inequality,  we obtain that
$$\io|\n u_n(x)|^\a dx\le C \hbox{ for all  } n.$$

As a consequence we get that $\{g_n\}_n$ is bounded in $L^{1+\e}(\O)$ for some $\e>0$. By the compactness result in Proposition \ref{key}, we obtain that, up to a
subsequence, $u_n\to u$ strongly in $W^{1,r}_0(\Omega)$ for all $r<p_*$ and $|\n u_n|\to |\n u|$ a.e.  in $\O$. Hence by Vitali lemma we reach that $u_n\to u$
strongly in $W^{1,\a}_0(\Omega)$ for all $\a<2s$, in particular, for $\a=q$. Thus $u$ is a solution to \eqref{qqA} with $u\in W^{1,\a}_0(\Omega)$ for all $a<2s$.
\end{proof}
As a consequence of the above Theorem and the construction of the supersolution at the beginning of this subsection, we get the following result.
\begin{Corollary} Assume that $q<2s$ and that $f\in L^\infty(\Omega)$ with $f\gneq 0$. Then there exists $\l^*>0$ such that for all $\l<\l^*$, problem
\eqref{qqA}
has a bounded positive solution $u$ such that $u\in W^{1,\a}_0(\Omega)$ for all $\a<2s$.
\end{Corollary}
\subsection{ A second existence result}
Now, as in the local case, we assume that  $f\in L^\g(\O)$ for some $\g>\dfrac{N}{q'(2s-1)}$,  $q'=\dfrac{q}{q-1}$. In order to  obtain a  solution, we need some
extra condition on the supersolution. We obtain the following result.
\begin{Theorem}\label{exit2}
Assume that $f\in L^\g(\O)$ for some $\g>\frac{N}{q'(2s-1)}$. Let $w$ be a nonnegative supersolution to \eqref{qqA} such that $w\in W^{1,\s}(\O)$ for some
$q<\s\le 2s$. Suppose that the following estimate holds,
\begin{equation}\label{condi1}
\sup_{x\in \O}\io \dfrac{w^{\s-1}(x)\mathcal{G}_s(x,y)}{|x-y|^{\s}}dx\le C.
\end{equation}
Then problem \eqref{qqA} has a solution $u$ such that $u\in W^{1,\s}_0(\O)$ and $T_k(u)\in H^s_0(\O)\cap W^{1, \a}_0(\O)$ for all $\a<2s$.
\end{Theorem}
\begin{proof} Define $\psi$, to be the solution to problem
\begin{equation}\label{inter0} \left\{
\begin{array}{rclll}
(-\Delta)^s \psi &= & \dfrac{w^{\s-1}(x)}{d^{\s}(x)} &\text{   in }\Omega , \\  \psi &=& 0 &\hbox{  in } \mathbb{R}^N\setminus\Omega.
\end{array} \right.
\end{equation}
Therefore  $\psi\in L^\theta(\O)$ for $\theta>\max\{\frac{\a}{q-\s}, \frac{N}{N-q'(2s-1)}\}$ if $q'(2s-1)<N$, i.e.,  $q>p_*$,  and $\psi\in L^\infty(\O)$ if $q\le
p_*$.

Let $u_n$ be the unique solution to the approximating problem \eqref{aprox1}, then  $u_n\le u_{n+1}\le w$ for all $n$. Since $w\in W^{1,\s}(\O)$,  we get the
existence of $u$ such that $u_n\uparrow u$ strongly in $L^{\s^*}(\O)$. As in the proof of the  Theorem  \ref{exit1}, setting $g_n(x)=\dfrac{|\n u_n|^p}{\frac
1n+|\n u_n|^p} +\l f$, it follows that
$$
u_n(x)=\io \mathcal{G}_s(x,y) g_n(y)dy\quad \hbox{    and then     } |\n u_n(x)|\le\io |\n_x \mathcal{G}_s(x,y)|g_n(y)dy.
$$
Therefore,
\begin{eqnarray*}
|\n u_n(x)|^{\s} &\le & \Big(\io |\n_x \mathcal{G}_s(x,y)|g_n(y)dy\Big)^{\s}\le \Big(\io h(x,y) \mathcal{G}_s(x,y)g_n(y)dy\Big)^{\s}\\& \le & \dyle \Big(\io
(h(x,y))^{\a} \mathcal{G}_s(x,y)g_n(y)dy\Big)\Big(\io \mathcal{G}_s(x,y)g_n(y)dy\Big)^{\s-1}\\ &\le & \dyle \Big(\io (h^{\s}(x,y)
\mathcal{G}_s(x,y)g_n(y)dy\Big)u^{\s-1}_n(x)\\ &\le & \dyle \io \Big(h^{\s}(x,y) \mathcal{G}_s(x,y)g_n(y)dy\Big)w^{\s-1}(x)\\ &\le & \dyle \io \Big(h^{\s}(x,y)
\mathcal{G}_s(x,y)|\n u_n(y)|^qdy\Big)w^{\s-1}(x)+ \l \io \Big(h^{\s}(x,y) \mathcal{G}_s(x,y)f(y)dy\Big)w^{\s-1}(x).
\end{eqnarray*}
Thus
\begin{eqnarray*}
\dyle\io |\n u_n|^{\s} dx & \le & \dyle \io|\n u_n(y)|^q \Big(\io h^{\s}(x,y) \mathcal{G}_s(x,y)w^{\s-1}(x)dx\Big)dy\\ & + & \dyle \l \io f(y)\Big(\io
h^{\s}(x,y)\mathcal{G}_s(x,y)w^{\s-1}(x)dx\big)dy\equiv J_1+J_2.
\end{eqnarray*}
Recall that $h(x,y)=\max\{\frac{1}{|x-y|},\frac{1}{d(x)}\}$, then

\begin{eqnarray*}
J_1 & \le & \dyle \io|\n u_n(y)|^q \Big(\io \dfrac{w^{\s-1}(x)G(x,y)}{|x-y|^{\s}}dx\Big)dy+ \io|\n u_n(y)|^q \Big(\io
\dfrac{w^{\s-1}(x)G(x,y)}{d^{\s}(x)}dx\Big)dy\\ & \le & \dyle \io|\n u_n(y)|^q \Big(\io \dfrac{w^{\s-1}(x)G(x,y)}{|x-y|^{\s}}dx\Big)dy+ \io|\n u_n(y)|^q
\psi(y)dy.
\end{eqnarray*}
By using the hypothesis on $w$, we reach that
\begin{eqnarray*}
J_1 &\le & C\io|\n u_n(y)|^q dy+\Big(\io|\n u_n(y)|^\s dy\Big)^{\frac{q}{\s}}\Big(\io \psi^{\frac{\s}{q-\s}}dy\Big)^{\frac{\s-q}{\s}}\\ &\le & \dyle C_1\io|\n
u_n(y)|^q dy+C_2\Big(\io|\n u_n(y)|^\s dy\Big)^{\frac{q}{\s}}.
\end{eqnarray*}
For $J_2$, we have
\begin{eqnarray*}
J_2  \le  \dyle \io f(y)\Big(\io \dfrac{w^{\s-1}(x)G(x,y)}{|x-y|^{\s}}dx\Big)dy+ \io f(y)\psi(y)dy.
\end{eqnarray*}
Hence
$$
J_2 \le  C\io f(y)dy+\Big(\io f^{\frac{N}{q'(2s-1)}}(y)dy\Big)^{\frac{q'(2s-1)}{N}}\Big(\io \psi^{\frac{N}{N-q'(2s-1)}}(y) dy\Big)^{\frac{N-q'(2s-1)}{N}}\le  C.
$$
Thus
$$
\io|\n u_n(x)|^\s dx\le C_1\io|\n u_n(x)|^qdx +C_2\Big(\io|\n u_n(y)|^\a dy\Big)^{\frac{q}{\a}}+C_3.
$$
Since $\s>q$, using H\"older inequality, it holds
$$\io|\n u_n(x)|^\s dx\le C \hbox{  for all }  n.$$
 Hence, up to a subsequence, $u_n\rightharpoonup u$ weakly in $W^{1,\s}_0(\O)$. By the compactness result in Proposition \ref{key},  up to a subsequence, we
 obtain that $|\n u_n|\to |\n u|$ a.e.  in $\O$. Hence by Vitali lemma, taking into account that $q<\s$, we reach that $u_n\to u$ strongly in $W^{1,q}_0(\O)$.
 Thus $u$ is a solution to \eqref{qqA} with $u\in W^{1,\s}_0(\O)$.  Define $F\equiv |\n u|^q+\l f$, then $F\in L^1(\O)$, hence using Theorem \ref{key01} it holds
 that $T_k(u)\in H^s_0(\O)\cap W^{1, \s}_0(\O)$ for all $\s<2s$. Hence we conclude.
 \end{proof}
As a consequence of the  Theorem \ref{exit2}, we get the following application  in  a concrete case.
\begin{Theorem}\label{exit3}
Assume that $p_*<q<2s$ and let $f(x)=\dfrac{1}{|x|^\theta}$ with $0<\theta<q'(2s-1)<N$. Then there exists $\l^*$ such that for all $\l<\l^*$, problem \eqref{qqA}
has a solution $u$ with $u\in W^{1,\s}_0(\O)$ for $q<\s<\min\{\frac{N}{\theta-2s+1}, 2s\}$ and $T_k(u)\in H^s_0(\O)\cap W^{1, \a}_0(\O)$ for all $\a<2s$.
\end{Theorem}
\begin{proof} From Theorem \ref{exit2}, we have just to build a supersolution $w$ such that $w\in W^{1,\s}(\O)$ for some $q<\s\le 2s$. It is clear that $f\in
L^\g(\O)$ for some $\g>\frac{N}{q'(2s-1)}$. Without loss of generality, we can assume that $\theta>2s$. Define $w_1(x)=\dfrac{1}{|x|^{\theta-2s}}$, then
$$(-\Delta)^s w_1=\frac{C}{|x|^\theta}\ge C_1(\O)|\n w_1|^q+C_2 f \mbox{  in   }\O.$$
Hence setting $w=cw_1$, we reach that, for small $\l$,
$$(-\Delta)^s w\ge |\n w|^q+\l f \mbox{  in   }\O. $$
It is clear that $w\in W^{1, \beta}(\O)$ for all $\beta<\frac{N}{\theta-2s+1}$. Since $q>p_*$, then $q<\frac{\theta}{\theta-2s+1}<\frac{N}{\theta-2s+1}$. Hence
there exists $q<\s<\min\{\frac{N}{\theta-2s+1}, 2s\}$ such that $w\in W^{1, \s}(\O)$. Moreover,  condition \eqref{condi1} holds.  Indeed,
$$\io \dfrac{w^{\s-1}(x)G(x,y)}{|x-y|^{\s}}dx\le C\io \dfrac{1}{|x|^{(\theta-2s)(\s-1)}|x-y|^{N-2s+\a_0}}dx.$$
Since $0<\theta<q'(2s-1)$, then we can choose $q<\s<2s$ such that $\s<\frac{\theta}{\theta-2s+1}$. That is, $\dfrac{1}{|x|^{(\theta-2s)(\s-1)}}\in L^\g(\O)$ for
some $\g>\frac{N}{2s-1}$. Therefore, using H\"older inequality, we obtain  that
$$\io \dfrac{w^{\s-1}(x)G(x,y)}{|x-y|^{\s}}dx\le C\Big(\io \dfrac{1}{|x|^{(\theta-2s)(\s-1)}})^\sigma dx\Big)^{\frac{1}{\g}}
\Big(\io \dfrac{1}{|x-y|^{\s'(N-2s+\s)}}dx\Big)^{\frac{1}{\g'}}.
$$
Since $\g'(N-2s+\s)<N$, then
$$\io \dfrac{w^{\s-1}(x)G(x,y)}{|x-y|^{\s}}dx\le C.$$
Define now $\psi$ to be the unique solution to the problem
\begin{equation}\label{inter00} \left\{
\begin{array}{rcll}
(-\Delta)^s \psi &= & \dfrac{w^{\s-1}(x)}{d^{\s}(x)}=\dfrac{1}{|x|^{(\theta-2s)(\s-1)}d^{\s}(x)} & \text{   in }\Omega , \\  \psi &=& 0 & \hbox{  in }
\mathbb{R}^N\setminus\Omega.
\end{array} \right.
\end{equation}
Since $(\theta-2s)(\s-1)<2s$, we  prove that $\psi\in L^\infty(\O)$.  Fix $B_r(0)\subset\subset \O$ and taking $\psi_1$ and $\psi_2$, the solutions to problems
\begin{equation}\label{psi1} \left\{
\begin{array}{rcll}
(-\Delta)^s \psi_1 &= & \dfrac{1}{|x|^{(\theta-2s)(\s-1)}d^{\s}(x)}\chi_{B_r(0)} &\text{   in }\Omega , \\  \psi_1 &=& 0 &\hbox{  in }
\mathbb{R}^N\setminus\Omega,
\end{array} \right.
\end{equation}
and
\begin{equation}\label{psi2} \left\{
\begin{array}{rcll}
(-\Delta)^s \psi_2 &= & \dfrac{1}{|x|^{(\theta-2s)(\s-1)}d^{\s}(x)}\chi_{\{\O\backslash B_r(0)\}} &\text{   in }\Omega , \\  \psi_2 &=& 0 &\hbox{  in }
\mathbb{R}^N\setminus\Omega,
\end{array} \right.
\end{equation}
it is clear that $\psi=\psi_1+\psi_2$.  Since $(\theta-2s)(\s-1)<2s$, then there exists $\s_1>\frac{N}{2s}$ such that
$\dfrac{1}{|x|^{(\theta-2s)(\s-1)}d^{\s}(x)}\chi_{B_r(0)}\in L^{\s_1}(\O)$. Therefore, $\psi_1\in L^\infty(\O)$. Respect to $\psi_2$, it is clear that
$\dfrac{1}{|x|^{(\theta-2s)(\s-1)}d^{\s}(x)}\chi_{\{\O\backslash B_r(0)\}}\le \dfrac{C}{d^{\s}(x)}$. Then, since $\s<2s$, using similar arguments as in the proof
of Lemma \ref{singular}, we reach that $\psi_2\in L^\infty(\O)$. Thus $\psi\in L^\infty(\O)$ and the claim follows.

Hence we conclude that all conditions of Theorem \ref{exit2} hold and therefore there exists $u$ a solution to problem \eqref{qqA} with $u\in W^{1,\s}_0(\O)$ and
$T_k(u)\in H^s_0(\O)\cap W^{1, \a}_0(\O)$ for all $\a<2s$.
\end{proof}
\begin{remark}
We do not reach the extremal case $\a=2s$. It is clear that by the previous monotonicity  method we can not reach the case $q\ge 2s$. However, in the next section
and using some arguments from  Potential  Theory we will show the existence of a solution if $q\ge 2s$.
\end{remark}
\section{Existence result using potential theory}\label{sec5}
In this section we will complete the above existence results for all $q>p_*$. We will use some techniques from potential theory. The  key is to construct  a
suitable supersolution using hypotheses on $f$ that allow us to use \textit{ potential theory estimates}.   In \cite{HMV} the authors prove the existence of
solution under potential type hypothesis on $f$ and for any $q\ge 1$ in the local setting. The hypothesis on $f$ is equivalent to  the condition \eqref{CC}. This
type of arguments was also used in \cite{BC} for $s=1$,  for   some potentials  instead of a gradient term.

In what follows we will assume that $f\in L^m(\O)$ with $m>\frac{N}{q'(2s-1)}$ and  then according with  the ideas of \cite{BC} and \cite{HMV}, we will build a
suitable supersolution to problem \eqref{qqA} in the whole space $\ren$ under natural conditions on $f$.

Consider the Riez potential $J_{(N-\a)}$ defined in Lemma \ref{ineq1}.  We call $I_\alpha=J_{(N-\a)}$, that is
  $$I_\a(g)(x)=\int_{\ren}\dfrac{g(y)}{|x-y|^{N-\a}}dy.$$
  Notice that if  $0<\a<2$, then $I_\a=((-\D)^{\frac{\a}{2}})^{-1}$ and $G_\a(x,y)\equiv \dfrac{1}{|x-y|^{N-\a}}$ is a constant multiple of the fundamental
  solution associated to the operator $(-\D)^{\frac{\a}{2}}$.

For $f\in L^m(\Omega)$, we consider its extension by $0$ to the whole $\mathbb{R}^N$, namely,
\begin{equation}\label{ff0}
f_0(x)= \left\{
\begin{array}{rcll}
f(x) & \mbox{  if } & x\in \Omega , \\  0 &\mbox{  if } &  x\in \mathbb{R}^N\setminus\Omega.
\end{array} \right.
\end{equation}
If $q$ is the exponent in the problem \eqref{qqA}, we define
$$F_0(x)=\Big(I_{2s-1}(f_0)(x)\Big)^q.$$
Then the key hypothesis on $f$ is that  the inequality
\begin{equation}\label{m00}
I_{2s-1}(F_0)\le C_1I_{2s-1}(f_0)
\end{equation}
holds.

The first  result of this Section is the following.
\begin{Theorem}\label{super}
Assume that the hypothesis \eqref{m00} holds, then problem \eqref{qqA} has a positive supersolution $\bar{u}$ such that $\bar{u}\in D^{1, \gamma}(\ren)$ where
$\gamma=\frac{mN}{N-m(2s-1)}\ge q$. In particular $u\in W^{1,q}_{loc}(\ren)$.
\end{Theorem}
\begin{proof} For the precise definition of the space $D^{1, \gamma}(\ren)$, the reader can consult Section 8.2 in \cite{Lieb-Loss}.
Assume that \eqref{m00} holds and define $\bar{u}=u_1+ u_2$ where $u_1, u_2$ solve the problems
\begin{equation}\label{prob1}
(-\Delta)^s u_1= \l f_0 \text{   in }\ren, (-\Delta)^s u_2= C_2 F_0(x)\text{   in }\ren,\quad \hbox{ respectively.}
\end{equation}
It is clear that
$$
\bar{u}(x)=I_{2s}(\l f_0+C_2 F_0)(x),
$$
thus
$$
|\n\bar{u}|\le (N-2s)\Big(I_{2s-1}(\l f_0+C_2 F_0)\Big).
$$
Hence using \eqref{m00}, we reach that
$$
|\n\bar{u}|^q\le \Big((N-2s)(\l+C_2C_1)\Big)^q\Big(I_{2s-1}(f_0)\Big)^q\le \Big((N-2s)(\l+C_2C_1)\Big)^q\, F_0.
$$
Therefore we conclude that
$$
(-\Delta)^s \bar{u}=\l f_0+C_2 F_0\ge \l f_0+\frac{C_2|\n\bar{u}|^q}{\Big((N-2s)(\l+C_2C_1)\Big)^q}.
$$
Let $\hat{u}=a\bar{u}$, where $a=\frac{C_2^{\frac{1}{q-1}}}{\Big((N-2s)(\l+C_2C_1)\Big)^{\frac{q}{q-1}}}$, then
\begin{equation}\label{TTT}
(-\Delta)^s \hat{u}\ge |\n\hat{u}|^q+ \l^* f_0 \mbox{  with  }\l^*= a\l.
\end{equation}
Since $|\n \hat{u}|\le CI_{2s-1}(f_0)$,    $f_0\in L^m(\ren)$  with  $m>\frac{N}{q'(2s-1)}$ and $q>p_*$, by using Lemma \ref{ineq1} and the Dominated Convergence
Theorem,  it follows that $|\n \hat{u}|\in L^\gamma(\ren)$ where $\gamma=\frac{mN}{N-m(2s-1)}\ge q$. Hence the result follows.
\end{proof}
In the next proposition we analyze  the regularity of the solution $\hat{u}$ obtained above.
\begin{proposition}\label{propo1}
Let $\hat{u}$ the supersolution obtained in Theorem \ref{super}.
\begin{enumerate}
\item If $m>\frac{N}{2s}$, then $\hat{u}\in L^\infty(\ren)$. \item If $\frac{N}{q'(2s-1)}<m\le \frac{N}{2s}$, then $\hat{u}\in L^{\frac{mN}{N-2sm}}(\ren)$,
    in
    particular for any $\alpha$ verifying $$1<\a<\a_0\equiv \frac{N}{N-m(2s-1)}<2s,$$ and for any bounded domain $\Omega$, we have
\begin{equation}\label{RA}
\sup_{\{y\in \O\}}\io \frac{\hat{u}^{\a-1}(x)}{|x-y|^{N-2s+\a}}dx\le C.
\end{equation}
\end{enumerate}
\end{proposition}
\begin{proof}  Consider $u_1$ and $u_2$ defined in the proof of Theorem \ref{super}.

Assume first that $m>\frac{N}{2s-1}$, then we get easily that $u_1\in L^\infty(\ren)$. Since
$$ F_0(x)=\big(\io \frac{f_0(y)}{|x-y|^{N-2s+1}}dy\big)^q,$$ using H\"older inequality, we reach that $F_0\in L^\infty(\ren)$.  Moreover, as a consequence of
Lemma \ref{ineq1}, $F_0\in L^1(\ren)$. Thus $F_0\in L^1(\ren)\cap L^\infty(\ren)$. Hence $u_2\in L^\infty(\ren)$ and then the result follows in this case.

If $\frac{N}{2s}<m\le \frac{N}{2s-1}$  then $u_1\in L^\infty(\ren)$.

Respect to $u_2$, we have $F_0\in L^{\gamma_1}(\ren)$ with $\gamma_1=\frac{mN}{q(N-m(2s-1)}$.
 A direct computation shows   that for $\frac{N}{2s}<m\le \frac{N}{2s-1}$, we have $\gamma_1>\frac{N}{q}$.

 Using again Lemma \ref{ineq1}, it holds that $F_0\in L^{\frac{2N}{N+2s}}(\ren)$. Thus $F_0\in L^{\sigma}(\ren)$ for all $\sigma \in [\frac{2N}{N+2s},
 \gamma_1]$.
 Hence $u_2\in L^\infty(\ren)$ and then $\hat{u}\in L^\infty(\ren)$.

 Let consider the hypothesis $(2)$, namely, $\frac{N}{q'(2s-1)}<m\le \frac{N}{2s}$.

 Since $|\n \hat{u}|\in L^\gamma(\ren)$ with $\gamma=\frac{mN}{N-m(2s-1)}$, by using the Sobolev inequality we conclude that $\hat{u}\in
 L^{\frac{mN}{N-2sm}}(\ren)$.

Let us prove now inequality \eqref{RA}. Fix $1<\a<\a_0$ and define
$$
K_1(y)=\io \frac{u_1^{\a-1}(x)}{|x-y|^{N-2s+\a}}dx\qquad\mbox{  and  } \qquad K_2(y)=\io \frac{u_2^{\a-1}(x)}{|x-y|^{N-2s+\a}}dx.
$$
Using H\"older's  inequality, we obtain
$$
K_i(y)\le \Big(\io u^{(\a-1)\s}_i(x)dx\Big)^{\frac{1}{\s}}\Big(\io \frac{1}{|x-y|^{(N-2s+\a)\s'}}dx\Big)^{\frac{1}{\s'}}\mbox{  for  }i=1,2.
$$
Since $\O$ is a bounded domain, then $\dyle\io \frac{1}{|x-y|^{(N-2s+\a)\s'}}dx\le C(\O)$ if $(N-2s+\a)\s'<N$, that means $\s>\frac{N}{2s-\a}$. Since $\a<\a_0$,
we can find $\s>\frac{N}{2s-\a}$ such that $(\a-1)\s\le \frac{mN}{N-2sm}$. Hence $K_1(y)\le C(u_1,\O)$ for all $y\in \O$.

We deal now with $K_2$.  By Lemma \ref{ineq1}, we get $F_0\in L^\theta(\ren)$ for $\theta=\frac{mN}{q(N-m(2s-1))}$. Therefore, using again Lemma \ref{ineq1}, we
find that $u_2\in L^{\theta_1}(\ren)$ where $\theta_1=\frac{mN}{q(N-m(2s-1))-2sm}$. Since $\O$ is bounded and  $\a<\a_0$, we can find a $\sigma$ such that
$\s>\frac{N}{2s-\a}$ and $\theta_1>(\a-1)\s$.  Hence
$$\io u^{(\a-1)\s}_2(x)dx\le C.$$
Finally,  as $\hat{u}=u_1+u_2$, we conclude.
\end{proof}

\begin{remark}

\

\begin{enumerate}
\item It is clear that the above computations allow us to get a  \textit{regular}  supersolution to problem  \eqref{qqA} for all $q>1$ under suitable
    hypothesis on $f$ and $\l$. \item If, moreover,  $q<p_*=\frac{N}{N-2s+1}$, then, as it was stated in Theorem \ref{compag}, a comparison principle holds
    in
    the sense that if $v$ is satisfies
\begin{equation}\label{qq000}
\left\{
\begin{array}{rcll}
(-\Delta)^s v &\le & |\nabla v|^{q}+\l g & \mbox{ in }\Omega , \\ v &\le & 0 &\hbox{  in } \mathbb{R}^N\setminus\Omega,
\end{array}%
\right.
\end{equation}%
with $(-\Delta)^s v, |\nabla v|^q\in L^1(\O)$ and $g\le f$, then $v\le \hat{u}$ in $\ren$. In the case where $g=0$, then from Theorem \ref{compag}, we know
that $v\le 0$ and then trivially $v\le \hat{u}$.

\item If $q\ge p_*$, then, in general non comparison principle holds. In the local case, namely $s=1$, it was proved in \cite{AA} that for $q\ge
    \frac{N}{N-1}$, the problem
    $$
    -\D u=|\n u|^q  \mbox{ in  }B_1(0), u\in W^{1,q}_0(\O),
    $$
has a non trivial solution $u$. Also it was proved in \cite{ADP} and \cite{AA} that there exists infinitely many positive non comparable solutions.
\end{enumerate}

\end{remark}

Now, as in \cite{HMV}, we can prove the following existence result.
\begin{Theorem}\label{super00}
Under the general hypotheses on $s$ and $q$, assume that $\Omega\equiv \mathbb{R}^N$. There exists a constant $C_1$ depending on $q$ and $N$  such that if
\eqref{m00}   holds with constant  $C_1$, then there exists a positive weak solution $u\in W^{1,q}_{loc}(\ren)$ to the equation
\begin{equation}\label{HMV}
(-\Delta )^s u=|\nabla u|^q+\l f,\quad f\ge 0 \hbox{      in    } \mathbb{R}^N.
\end{equation}
\end{Theorem}
\begin{proof}
Take  $u_1$ as the solution to problem
\begin{equation*}
(-\Delta)^s u_1=  \l f \text{   in      }\ren.
\end{equation*}
 and define by recurrence $u_{k+1}$ by setting
$$
u_{k+1}=I_{2s}(|\n u_k|^q)+ I_{2s}(\l f).
$$
Then
$$
(-\Delta)^s u_{k+1}=|\n u_k|^q+  \l f \text{   in }\ren.
$$
We claim that
\begin{equation}\label{one10}
|\n u_k|\le C_1I_{2s-1}(\l f),
\end{equation}
\begin{equation}\label{two10}
|\n u_{k+1}-\n u_k|\le C_2\d^k I_{2s-1}(\l f) \mbox{ for some }\d<1,
\end{equation}
and the sequence $\{u_k\}_k$ is a Cauchy sequence in the space $D^{1, \g}(\ren)$ where $\gamma=\frac{mN}{N-m(2s-1)}\ge q$.

We prove first \eqref{one10} arguing by induction. It is clear that \eqref{one10} holds for $k=1$. Assume that
$$
|\n u_k|\le a_k I_{2s-1}(f),
$$
then we know that
\begin{eqnarray*}
|\n u_{k+1}| &= & |\n (I_{2s}(|\n u_k|^q))+ \n (I_{2s}( f))|\\ &\le & C\bigg(I_{2s-1}(|\n u_k|^q)+I_{2s-1}( f)\bigg)\\ &\le & C \bigg(a_k^q
I_{2s-1}((f_0)^q)+I_{2s-1}( f)\bigg).
\end{eqnarray*}
Now using \eqref{m00}, we conclude that
$$
|\n u_{k+1}|\le a_{k+1}I_{2s-1}(f),
$$
with
$$
a_{k+1}=C(C_1 a_k^q+1).
$$
Then, if  $C_1\le \dfrac{q'^{1-q}}{q C^q}$,
$$\lim_{k\to\infty} a_k=a \le Cq',$$
where $a$  the smaller root to the equation $\tau=C(C_1\tau^q+1)$. Hence \eqref{one10} follows.

As a conclusion and using the same computation as in the last part of the proof of Theorem \ref{super} we reach that the sequence $\{u_k\}_k$ is bounded in the
space $D^{1, \g}(\ren)$ where $\gamma=\frac{mN}{N-m(2s-1)}$. Now using estimate \eqref{one10} we reach that $\{u_k\}_k$ is a Cauchy sequence in $D^{1, \g}(\ren)$.
Hence there exists  $u\in D^{1,\gamma}(\ren)$ such that $u_k\to u$ strongly in $D^{1, \g}(\ren)$. Since $q<\gamma$, then $u_k\to u$ strongly in
$W^{1,q}_{loc}(\ren)$ and then $u$ solves \eqref{qqA} at least in the sense of distributions.
\end{proof}

\subsection{The subcritical case $q<2s$}\label{sub1}
In this subsection we consider the case $q<2s$. We will combine the above ideas in order to show the existence of a  \textit{suitable supersolution} under
\textit{natural} condition on $f$. Then using the comparison principle and the representation formula, as in the previous section, we show the existence of a
minimal solution. More precisely we have the next result.
\begin{Theorem}\label{exit4}
Assume that $q<2s$. Suppose that $f\in L^m(\O)$ where $m>\frac{N}{q'(2s-1)}$ and the hypothesis \eqref{m00} holds. There there exists $\l^*>0$ such that for all
$0<\lambda<\l^*$  problem \eqref{qqA} has  a solution $u$ such that $u\in W^{1,q}_0(\O)$ and $T_k(u)\in H^s_0(\O)\cap W^{1, \s}_0(\O)$ for all $\s<2s$.
\end{Theorem}
\begin{proof}
We follow closely the argument used in the proof of Theorem \ref{exit2}. Let $u_n$ to be the unique solution to the approximated problem \eqref{aprox1}, then
$u_n\le u_{n+1}$. Fix $\l<\l^*$ defined in \eqref{TTT} and let $\hat{u}$ be the supersolution obtained in Theorem \ref{super}. It is clear that $\hat{u}$ is a
supersolution to problem \eqref{aprox1}. Hence by the comparison principle in Theorem \eqref{compa2} we reach that $u_n\le \hat{u}$ for all $n$. Hence following
the same computation as in the proof of Theorem \ref{exit1}, we obtain that
\begin{equation}\label{main001}
\begin{array}{lll}
\dyle\io |\n u_n(x)|^{\a} dx & \le & \dyle \io|\n u_n(y)|^q \Big(\io h^{\a}(x,y) \mathcal{G}_s(x,y)\hat{u}^{\a-1}(x)dx\Big)dy\\ &+& \l \dyle\io f(y)\Big(\io
h^{\a}(x,y)\mathcal{G}_s(x,y)\hat{u}^{\a-1}(x)dx\big)dy.
\end{array}
\end{equation}
Now, we divide the proof into two cases, according to the value of $m$ and the regularity of $\hat{u}$.

{\bf{The first case: $\frac{N}{2s}<m$.}} In this case, using Proposition \ref{propo1}, we know that $\hat{u}\in L^\infty(\Omega)$. Hence, following again the
proof of Theorem \ref{exit1}, and taking into consideration \eqref{main001}, we conclude that $\dyle\io |\n u_n|^{\a} dx\le C$ for all $n$ provided that $\a<2s$.
Now the rest of the proof follows exactly the proof of Theorem \ref{exit1}.

{\bf{The second case: $\frac{N}{q'(2s-1)}<m\le \frac{N}{2s}$.}} Since $q\ge p_*$, then using Lemma \ref{ineq1}, we obtain that $|\n \hat{u}|\in L^\gamma(\O)$
where $\gamma=\frac{mN}{N-m(2s-1)}>q$.

We set
$$
K(y)=\io h^{\a}(x,y) \mathcal{G}_s(x,y)\hat{u}^{\a-1}(x)dx,
$$
then by \eqref{main001}, we have
\begin{equation}\label{main0011}
\dyle\io |\n u_n(x)|^{\a}dx \le \dyle \io|\n u_n(y)|^q K(y)dy+\l \dyle\io f(y)K(y)dy.
\end{equation}
We claim that for $q<\a<\a_0$, defined in Proposition \ref{propo1}, we have $K\in L^\infty(\O)$. Notice that
$$
K(y)\le\io \dfrac{\hat{u}^{\a-1}(x)G(x,y)}{|x-y|^{\a}}dx+ \io \dfrac{\hat{u}^{\a-1}(x)G(x,y)}{d^{\a}(x)}dx=J_1+J_2.
$$
It is clear that
$$
J_1\le \io \dfrac{\hat{u}^{\a-1}(x)}{|x-y|^{N-2s+\a}}dx.
$$
Hence by the second point in Proposition \ref{propo1}, we obtain that $\dyle\io\dfrac{\hat{u}^{\a-1}(x)}{|x-y|^{N-2s+\a}}dx\le C$. Thus $J_1(y)\le C$. To analyze
$J_2$, we follow the same computation as in  the proof of Theorem \ref{exit1}. As a consequence, we reach that $J_2\le C$.

Therefore we conclude that $K(y)\le C$ for all $y\in \O$ and the claim follows.

 Following again the last part of the proof of Theorem \ref{exit1}, we get the existence of a solution $u$ to problem \eqref{qqA} such that $u\le \hat{u}$, $u\in
 W^{1,\a}_0(\O)$ for all $\a<\a_0$ and $T_k(u)\in H^s_0(\O)\cap W^{1, \s}_0(\O)$ for all $\s<2s$. It is clear that $u$ is a minimal solution to \eqref{qqA}.
\end{proof}

We give now a capacity-based condition   on $f$ in order to show that condition \eqref{m00} holds. Let recall the following result proved in \cite{MV}.
\begin{Theorem}\label{th00}
Assume that $f\in L^1(\ren)$ is a nonnegative function. For any compact set $E\subset \ren$, we define $\dyle |E|_f=\int_Ef(y)dy$.

Then $f$ satisfies the condition \eqref{m00},  if and only if, for any compact set $E\subset \ren$,
\begin{equation}\label{YY}
|E|_f\le C\,\text{cap}_{2s-1,q'}(E)
\end{equation}
where
$$
\text{cap}_{2s-1,q'}(E)\equiv \inf\bigg\{\int_{\ren}\int_{\ren} \dfrac{|\phi(x)-\phi(y)|^{q'}}{|x-y|^{N+q'(2s-1)}}dxdy \mbox{  with }\phi\in
\mathcal{C}^\infty_0(\ren)\mbox{  and  }\phi\ge \chi_E\bigg\}.
$$
\end{Theorem}
As a consequence we have the following result.
\begin{Theorem}\label{th011}
Let $f\in L^m(\O)$ with $m\ge \frac{N}{q'(2s-1)}$ and define $f_0$ as in \eqref{ff0}. Then the condition \eqref{m00} holds for $f_0$.
\end{Theorem}
\begin{proof}
We have just to show that the condition \eqref{YY} holds. Let $E$ be a compact set and consider $\phi\in \mathcal{C}^\infty_0(\ren)$ be such that $\phi\ge
\chi_E$. Using H\"older inequality
$$
|E|_{f_0}\le \int_E f_0|\phi|^{q'}dx\le \Big(\int_{\ren} f_0^{\frac{p^*_{s_1}}{p^*_{s_1}-p}}(x)dx\Big)^{\frac{p^*_{s_1}-p}{p^*_{s_1}}} \Big(\int_{\ren}
|\phi(x)|^{p^*_{s_1}}dx\Big)^{\frac{p}{p^*_{s_1}}}
$$
where $p=q', s_1=2s-1$ and $p^*_{s_1}=\frac{pN}{N-ps_1}$. It is clear that $\frac{p^*_{s_1}}{p^*_{s_1}-p}=\frac{N}{q'(2s-1)}$. Since $m\ge \frac{N}{q'(2s-1)}$,
using Sobolev inequality, it holds that
$$
|E|_{f_0}\le C(\O)||f||_{L^m(\O)}\text{cap}_{2s-1,q'}(E).
$$
Hence we conclude.
\end{proof}
\subsection{The critical case $q= 2s$}\label{critical} In this case there are difficulties to use the comparison arguments.  It is possible to find  a
supersolution but it is not clear how to pass to the limit in the gradient term when dealing with the family of approximating problems.

In the local case this difficulty is overpasses by using convenient nonlinear test functions and a suitable change of variable. In the nonlocal framework it seems to be necessary to change this point of view and to adapt a different approach. We will use in a convenient way the Schauder fixed point theorem following the
strategy used recently  in \cite{CON1} and \cite{CON2} for the local case, under suitable condition on $s$.

{
\begin{Theorem}\label{fix00}
Suppose that $\O$ is a bounded regular domain and that $f\in L^m(\O)$ where $m>\frac{N}{2s}$. Then there exists $\l^*(f)>0$ such that for all $\l<\l^*$, problem
\eqref{qqA} has a solution $u\in W^{1,2s}_{loc}(\O)\cap W^{1,1}_0(\O)$.
\end{Theorem}
}
\begin{proof} Without loss of generality we can assume that $N\ge 2$. Suppose that $f\in L^m(\O)$ where $m>\frac{N}{2s}$. Lemma \ref{key2-corr} and Remark \ref{loc-reg}, we know that if $v$ is a solution to  problem \eqref{gener1-corr}, then for all
$p<\frac{m N}{N-m(2s-1)}$, there exists a positive constant $C_0$ such that
\begin{equation}\label{fix2}
|||\n v|\,d^{1-s}||_{L^p(\Omega)}\le C_0||f||_{L^m(\O)},
\end{equation}
and
\begin{equation}\label{fix25}
|||\n (v\,d^{1-s})||_{L^p(\Omega)}\le \hat{C}_0||f||_{L^m(\O)}.
\end{equation}

Since $m>\frac{N}{2s}$, then $\s_0\equiv 2sm<\frac{m N}{N-m(2s-1)}$. Now, taking into account that $2s>1$, then for a universal constant $\bar{C}$ that we will chose later,  we get the existence of $\l^*>0$ such that for some $l>0$, we
have
$$
\bar{C}(l+\l^*||f||_{L^m(\O)})=l^{\frac{1}{2s}}.
$$
Fix $\l<\l^*$ and  $l>0$ as above. We define the set
\begin{equation}\label{sett}
E=\{v\in W^{1,1}_0(\O): v\, d^{1-s}\in W^{1,2sm}_0(\O)\mbox{  and  } \bigg(\io |\n (v\, d^{1-s})|^{2sm} dx\bigg)^{\frac{1}{2sm}}\le l^{\frac{1}{2s}}\},
\end{equation}
It is easy to check that $E$ is a closed convex set of $W^{1,1}_0(\O)$.
Notice that, using Hardy inequality, we obtain that if $v\in E$, then $|\n v|^{2sm}d^{2sm(1-s)}\in L^1(\O)$ and
$$
\bigg(\io |\n v|^{2sm}\, d^{2sm(1-s)}dx\bigg)^{\frac{1}{2sm}}\le \hat{C}_0 l^{\frac{1}{2s}}.
$$
Consider the operator
$$
\begin{array}{rcl}
T:E &\rightarrow& W^{1,1}_0(\O)\\
   v&\rightarrow&T(v)=u
\end{array}
$$
where $u$ is the unique solution to problem
\begin{equation}
\left\{
\begin{array}{rcll}
(-\Delta)^s u &= & |\nabla v|^{2s}+\l f & \text{ in }\Omega , \\ u &=& 0 &\hbox{  in } \mathbb{R}^N\setminus\Omega,\\ u&>&0 &\hbox{ in }\Omega.
\end{array}%
\right.  \label{fix1}
\end{equation}
{\sc First step:}
We prove that $T$ is well defined. By the result of Theorem \ref{key} and Remark \ref{rm01}, to get the desired result, we just have to show the existence of $\b<2s-1$ such that $|\n v|^{2s}d^\b\in L^1(\O)$. It is cleat that $|\n v|^{2s}\in L^1_{loc}(\O)$, moreover, we have
$$\begin{array}{ll}
&\dyle\io|\n v|^{2s}d^\b dx=\io|\n v|^{2s}d^{2s(1-s)}d^{\beta-2s(1-s)}dx\\
&\dyle\le \bigg(\io|\n v|^{2sm}d^{2s(1-s)m}dx\bigg)^{\frac{1}{m}}
\bigg(\io d^{(\beta-2s(1-s))m'}dx\bigg)^{\frac{1}{m'}}.
\end{array}
$$
If $2s(1-s)<2s-1$, we can chose $\beta<2s-1$ such that $2s(1-s)<\beta$. Hence
$$\io d^{(\beta-2s(1-s))m'}dx<\infty.$$

Assume that $2s(1-s)\ge 2s-1$, then $s\in (\frac{1}{2}, \frac{1}{\sqrt{2}}]$. Notice that $2s(1-s)-(2s-1)=1-2s^2$. Since $m>\frac{N}{2s}$ and $N\ge 2$, then $(1-2s^2)m'<1$. Hence we get easily the existence of $\beta<2s-1$ such that
$(2s(1-s)-\beta)m'<1$ and then we conclude.

Then using the fact that $v\in E$, we reach that
$|\nabla v|^{2s}d^\b+\l f \in L^1(\O)$. Therefore the existence of $u$ is a consequence of Theorem \ref{key} and Remark \ref{rm01}. Moreover,  $|\n u|\in L^{\a}(\O)$ for all
$\a<\frac{N}{N-s}$. Hence $T$ is well defined.

{\sc Second step:}

We claim that:
\begin{enumerate}
\item For $\delta>0$ small enough, $T(E)\subset E$, \item $T$ is a continuous and  compact operator on $E$.
\end{enumerate}
\textsc{Proof of (1)}.

We have
$$
u(x)=\io \mathcal{G}_s(x,y)|\n v(y)|^{2s}dy+\l \io \mathcal{G}_s(x,y)f(y))dy,
$$
then
$$
u(x)\le C\bigg(\io \frac{|\n v(y)|^{2s}d^{\theta s}(y) d^{(1-\theta)s}(x)}{|x-y|^{N-s}}dy+\l \io \frac{d^{\theta s}(y) d^{(1-\theta)s}(x) f(y)}{|x-y|^{N-s}}dy\bigg).
$$
Setting $\theta=2(1-s)<1$, it holds that
$$
u(x)\le C d^{s(2s-1)}(x)\bigg(\io \frac{|\n v(y)|^{2s}d^{2s(1-s)}(y)}{|x-y|^{N-s}}dy+\l \io \frac{d^{2s(1-s)}(y) f(y)}{|x-y|^{N-s}}dy\bigg).
$$
Since $v\in E$, then by Lemma \ref{ineq1}, we obtain that $\dfrac{u}{d^{s(2s-1)}}\in L^{\frac{mN}{N-ms}}(\O)$ and
$$
||\frac{u}{d^{s(2s-1)}}||_{L^{\frac{mN}{N-ms}}(\O)}\le C\bigg(||\,|\n v|^{2s}d^{2s(1-s)}\,||_{L^m(\O)}+\l ||f||_{L^m(\O)}\bigg)\le C(l^{\frac{1}{2}}+\l ||f||_{L^m(\O)}).
$$

On the other hand, we have
\begin{eqnarray*}
|\n u(x)| &\le & \io |\n_x \mathcal{G}_s(x,y)| (|\n v(y)|^{2s} \l f(y))dy\\
&\le & \io \frac{|\n_x \mathcal{G}_s(x,y)|}{\mathcal{G}_s(x,y)}\mathcal{G}_s(x,y) (|\n v(y)|^{2s} dy+ \l f(y))dy.
\end{eqnarray*}
Recall that form \ref{green2}, we have
$$
|\n_x \mathcal{G}_s(x,y)|\le C_2 \mathcal{G}_s(x,y)\max\{\frac{1}{|x-y|}, \frac{1}{d(x)}\},
$$
thus
\begin{eqnarray*}
|\n u(x)| &\le & C_2\int_{\{|x-y|<d(x)\}} \frac{\mathcal{G}_s(x,y)}{|x-y|}|\n v(y)|^{2s} dy+ \frac{C_2}{d(x)} \int_{\{|x-y|\ge d(x)\}} \mathcal{G}_s(x,y)|\n v(y)|^{2s} dy \\
&+ & C_2\l \bigg(\int_{\{|x-y|<d(x)\}} \frac{\mathcal{G}_s(x,y)}{|x-y|}f(y)dy+\frac{1}{d(x)}\int_{\{|x-y|\ge d(x)\}} \mathcal{G}_s(x,y)f(y)dy\Bigg).
\end{eqnarray*}
Hence
\begin{eqnarray*}
|\n u(x)| d^{1-s}&\le & C_2(I_1(x)+I_2(x)+J_1(x)+J_2(x)),
\end{eqnarray*}
where
$$
I_1(x)=d^{1-s}(x)\int_{\{|x-y|<d(x)\}} \frac{\mathcal{G}_s(x,y)}{|x-y|}|\n v(y)|^{2s} dy,
$$
$$
I_2(x)=\frac{1}{d^s(x)} \int_{\{|x-y|\ge d(x)\}} \mathcal{G}_s(x,y)|\n v(y)|^{2s} dy,
$$
$$
J_1(x)=d^{1-s}(x)\int_{\{|x-y|<d(x)\}} \frac{\mathcal{G}_s(x,y)}{|x-y|}f(y)dy,
$$
and
$$
J_2(x)=\frac{1}{d^s(x)} \int_{\{|x-y|\ge d(x)\}} \mathcal{G}_s(x,y)f(y) dy.
$$
Thus to prove the claim, we will show separately that $I_i,J_i\in L^{2sm}(\O)$ for $i=1,2$.

Let us begin by estimating $I_1$.

We have

\begin{eqnarray*}
I_1(x) &=& d^{1-s}(x)\int_{\{\frac 12 d(x)\le |x-y|<d(x)\}} \frac{\mathcal{G}_s(x,y)}{|x-y|}|\n v(y)|^{2s} dy
+ d^{1-s}(x)\int_{\{|x-y|<\frac 12 d(x)\}} \frac{\mathcal{G}_s(x,y)}{|x-y|}|\n v(y)|^{2s} dy\\
&\le & I_{11}(x)+I_{12}(x).
\end{eqnarray*}
Since
$$
\mathcal{G}_s(x,y)\le C\frac{d^{1-s}(y)d^{2s-1}(x)}{|x-y|^{N-s}},
$$
we obtain that
$$
I_{11}(x)\le C\int_{\{\frac 12 d(x)\le |x-y|<d(x)\}} \frac{|\n v(y)|^{2s}d^{1-s}(y)}{|x-y|^{N-(2s-1)}} dy\le C\io\frac{|\n v(y)|^{2s}d^{1-s}(y)}{|x-y|^{N-(2s-1)}} dy.
$$
Since $|\n v|^{2s}d^{1-s}\in L^{m}(\O)$, using Lemma \ref{ineq1}, it holds that $I_{11}\in L^{\frac{mN}{N-(2s-1)m}}(\O)$. As $m>\frac{N}{2s}$, then $\frac{mN}{N-(2s-1)m}>2sm$. Hence we conclude that
$$
||I_{11}||_{L^{2sm}(\O)}\le C|||\n v|^{2s}d^{1-s}||_{L^{2sm}(\O)}.
$$

We deal now with $I_{12}$. Since $d$ is a Lipschitz function, it holds that, for $(x,y)\in \{|x-y|<\frac 12 d(x)\}$,
$$
|d(x)-d(y)|\le |x-y|\le \frac 12 d(x).
$$
Thus
$$
\frac 12 d(x)\le d(y)\le \frac 32 d(x).
$$
Hence we obtain that
$$
I_{12}(x)\le Cd^{1-s}(x)\int_{\{|x-y|<\frac 12 d(x)\}} \frac{|\n v(y)|^{2s}}{|x-y|^{N-(2s+1)}}dy\le C\int_{\{|x-y|<\frac 12 d(x)\}} \frac{|\n v(y)|^{2s}d^{1-s}(y)}{|x-y|^{N-(2s+1)}}dy.
$$
Therefore, as in the estimate of $I_{11}$, we reach that
$$
||I_{12}||_{L^{2sm}(\O)}\le C|||\n v|^{2s}d^{1-s}||_{L^{2sm}(\O)}.
$$
We treat now $I_2$ which is more involved.
We have
\begin{eqnarray*}
I_2(x) &\le & \frac{1}{d^s(x)} \int\limits_{\{d(x)\le |x-y|\le d(y)\}} \mathcal{G}_s(x,y)|\n v(y)|^{2s} dy+
\frac{1}{d^s(x)} \int\limits_{\{|x-y|\ge \max\{d(y),d(x)\}\}} \mathcal{G}_s(x,y)|\n v(y)|^{2s} dy\\
&=& I_{21}(x)+I_{22}(x).
\end{eqnarray*}
Respect to $I_{21}$, we have
\begin{eqnarray*}
I_{21}(x) &\le & C\frac{1}{d^s(x)} \int_{\{d(x)\le |x-y|\le d(y)\}} \frac{|\n v(y)|^{2s} d^s(x)}{|x-y|^{N-s}}dy\le C\int_{\{d(x)\le |x-y|\le d(y)\}} \frac{|\n v(y)|^{2s}}{|x-y|^{N-s}}dy\\
&\le & C \int_{\{d(x)\le |x-y|\le d(y)\}} \frac{|\n v(y)|^{2s}d^{1-s}(y)}{|x-y|^{N-s} d^{1-s}(y)}dy\\
&\le & C \int_{\{d(x)\le |x-y|\le d(y)\}} \frac{|\n v(y)|^{2s}d^{1-s}(y)}{|x-y|^{N-(2s-1)}}dy.
\end{eqnarray*}
Hence, as in the estimate of $I_{11}$, using again Lemma \ref{ineq1}, it holds that
$$
||I_{21}||_{L^{2sm}(\O)}\le C|||\n v|^{2s}d^{1-s}||_{L^{2sm}(\O)}.
$$
We deal now with $I_{22}$. Notice that
$$
\mathcal{G}_s(x,y)\le C\frac{d^s(x)d^s(y)}{|x-y|^N},
$$
Hence
\begin{eqnarray*}
I_{22}(x) &\le & C \int_{\{|x-y|\ge \max\{d(y),d(x)\}\}} \frac{|\n v(y)|^{2s}d^s(y)}{|x-y|^N}dy\\
&\le & C \int_{\{|x-y|\ge \max\{d(y),d(x)\}\}} \frac{|\n v(y)|^{2s}d^{1-s}(y)d^{2s-1}(y)}{|x-y|^N}dy\\
&\le & C \int_{\{|x-y|\ge \max\{d(y),d(x)\}\}} \frac{|\n v(y)|^{2s}d^{1-s}(y)}{|x-y|^{N-(2s-1)}}dy.
\end{eqnarray*}
Hence we conclude as in the previous estimates.

Thus, as a conclusion we have proved that
$$
||I_{1}||_{L^{2sm}(\O)}+ ||I_{2}||_{L^{2sm}(\O)}\le C|||\n v|^{2s}d^{1-s}||_{L^{2sm}(\O)}.
$$
Respect to $J_1, J_2$, since $f\in L^m(\O)$, then we can repeat the same decomposition  as above and we reach the same conclusion.

Therefore, we find that
\begin{equation}\label{uu}
\| |\n u| d^{1-s}\|_{L^{2sm}(\O)}\le C\bigg(\| |\n v|^{2s}d^{1-s}\|_{L^{2sm}(\O)}+\l ||f||_{L^m(\O)}\bigg)\le l^{\frac{1}{2s}}.
\end{equation}
Hence the first point follows.

\textsc{ Proof of (2)}. To show the continuity of $T$ respect to the topology of $W^{1,1}_0(\O)$, we consider $\{v_n\}_n\subset E$ such that $v_n\to v$ strongly
in $W^{1,1}_0(\O)$. Define $u_n=T(v_n), u=T(v)$.

We have to show that $u_n\to u$ strongly in $W^{1,1}_0(\O)$;  by Theorem \ref{key}, to do this,  we will prove that  $$||\n v_n -\n v||_{_{L^{2s}(d^\beta dx, \O)}}\to 0 \hbox{   as } n\to \infty,$$ for some $\beta<2s-1$. As in the proof of the fist step, we get the existence of $\beta<2s-1$ such that
$$\io|\n v_n|^{2s}d^\b dx\le C \hbox{   for all   } n.$$
Since $2sm>1$, then setting $a=\dfrac{2s(m-1)}{2sm-1}<1$, it follows that $\dfrac{2s-a}{1-a}=2sm$. Hence by Hölder inequality, we conclude that
\begin{eqnarray*}
||\n v_n -\n v||_{_{L^{2s}(d^\beta dx, \O)}} & \le & ||\n v_n-\n v||^{\frac{a}{2s}}_{L^{1}(d^{\beta} dx, \O)} ||\n v_n-\n v||^{\frac{2s-a}{2s}}_{L^{\frac{2s-a}{1-a}}(d^{1-s} dx, \O)}\\ &\le & C||\n v_n-\n
v||^{\frac{a}{2s}}_{L^{1}(\O)}\to 0\mbox{  as  }n\to \infty.
\end{eqnarray*}
Now, using the definition of $u_n$ and $u$, there results that $u_n\to u$ strongly in $W^{1,1}_0(\O)$. Thus $T$ is continuous.

To finish we have just to show that $T$ is compact respect to the topology of $W^{1,1}_0(\O)$.

Let $\{v_n\}_n\subset E$ be such that $||v_n||_{W^{1,1}_0(\O)}\le C$. Since $\{v_n\}_n\subset E$, then
 $$||\n (v_n d^{1-s})||_{L^{2sm}(\O)}\le C$$
  and therefore up to a subsequence, $v_{n_k}d^{1-s}\rightharpoonup vd^{1-s}$ weakly in $W^{1,2sm}_0(\O)$. It is clear that $v_{n_k}\rightharpoonup v$ weakly in $W^{1,2sm}_{loc}(\O)$

Define
$$
F_n=|\nabla v_n|^{2s}+\l f, F=|\nabla v|^{2s}+\l f,
$$
then, as in the first step, $F_n d^\beta$ is bounded in $L^{1+\d}(\O)$ and $F_n d^\beta\rightharpoonup F d^\beta$ weakly in $L^{1+\d}(\O)$ for some $\d>0$ and $\beta<2s-1$. Using the compactness result of \cite{CV1}, we conclude
that, up to a subsequence, $u_{n_k}\to u$ strongly in $W^{1,1}_0(\O)$, hence the claim follows.

As a conclusion and using the Schauder Fixed Point Theorem, there exists $u\in E$ such that $T(u)=u$, then $u\in W^{1,2s}_0(\O)$ and $u$ solves \eqref{qqA}.
\end{proof}
\begin{remark}\label{rm11}

\

\begin{enumerate}
\item The solution obtained above is the unique solution in $E$.  Indeed, assume $u_1$ and $u_2\in E$ solutions to problem \eqref{qqA}.  Therefore  $||\n (u_1 d^{1-s})||_{L^{2sm}(\O)}<\infty$ and $ ||\n (u_2 d^{1-s})||_{L^{2sm}(\O)}<\infty$.

    Define $w=u_1-u_2$, then $||\n (w d^{1-s})||_{L^{2sm}(\O)}<\infty$, $w\in W^{1,\theta}_0(\O)$ for all $\theta<\frac{N}{N-2s+1}$ and $w$
    solves the problem
$$
\left\{
\begin{array}{rcll}
(-\Delta)^s w &= & |\nabla u_1|^{2s}-|\nabla u_2|^{2s} & \text{ in }\Omega , \\ w &=& 0 &\hbox{  in } \mathbb{R}^N\setminus\Omega.
\end{array}%
\right.
$$
Setting $b(x)=|\nabla u_1|^{2s-1}+|\nabla u_2|^{2s-1}$, the following inequality holds
$$
(-\Delta)^s w \le 2s b(x)|\n w|\text{ in }\Omega.
$$
Since $0<2s-1<1$, then $b\in L^1(\O)$.  Then as $m>\frac{N}{2s}$, we get $b\in L^\s(\O)$ for $\s>\frac{N}{2s-1}$.
Therefore, using the comparison principle in Theorem \ref{compa2}, it follows that $w_+=0$. Thus
$u_1\le u_2$. In a similar way we get $u_2\le u_1$. Hence $u_1=u_2$.

\item The solution $u\in E$ is the minimal solution to problem \eqref{qqA}. Assume that $v$ is an other solution to \eqref{qqA} with $|\n v|^{es}d^s\in
L^1(\O)$. As above, setting $w=(u-v)_+$, using the fact that for all $\xi_1, \xi_2\in \ren$, for all $\alpha>1$, we have
$$
|\xi_1|^\alpha-|\xi_2|^\alpha\le \alpha|\xi_1|^{\alpha-2}\langle\xi_1, \xi_1-\xi_2\rangle,
$$
and by Kato inequality, it follows that
$$
\left\{
\begin{array}{rcll}
(-\Delta)^s w &\le & |\nabla u|^{2s}-|\nabla v|^{2s}\le 2s|\n u|^{2s-1}|\n w| & \text{ in }\Omega , \\ w &=& 0 &\hbox{  in } \mathbb{R}^N\setminus\Omega.
\end{array}%
\right.
$$
Setting $b(x)=2s|\nabla u|^{2s-1}$ and using the fact that $u\in E$, there results that $b\in L^1(\O)\cap L^\s_{loc}(\O)$ for $\s=\frac{2sm}{2s-1}>\frac{N}{2s-1}$. As above,
using the comparison principle in Theorem \ref{compa2}, we conclude that $w=0$. Hence $u\le v$.
\end{enumerate}
\end{remark}
\subsection{The supercritical case $q>2s$}\label{sub31}
According to the value of $q$, we will consider two cases: $2s<q\le \dfrac{s}{1-s}$ and $q>\frac{s}{1-s}$.

In the first case and in a similar way as in  the critical case $q=2s$, we prove the following  result.
\begin{Theorem}\label{fix011}
Suppose that $\O$ is a bounded regular domain and $2s<q\le \frac{s}{1-s}$. Assume that $f\in L^m(\O)$ with $m>\dfrac{N}{q'(2s-1)}$. Then there exists $\l^*(f)>0$ such that for all $\l<\l^*$,
problem \eqref{qqA} has a solution $u$ such that $u d^{1-s}\in W^{1,qm}_0(\O)$.
\end{Theorem}
\begin{proof} In this case we choose $l>0$ and $\s_0$ such that
$$
\s_0\equiv 2sm<\frac{m N}{N-m(2s-1)}\mbox{  and   } C_0(l+\l^*||f||_{L^m(\O)})=l^{\frac{1}{q}}.
$$
Define the set
\begin{equation}\label{sett2}
E=\{v\in W^{1,1}_0(\O): v d^{1-s}\in W^{1,qm}_0(\O)\mbox{  and  }||\n (v d^{1-s}||_{L^{qm}(\O)}\le l^{\frac{1}{q}}\}.
\end{equation}
As in the proof of Theorem \ref{fix00} we consider $T_q:E \to W^{1,1}_0(\O)$ defined by $u=T_q(v)$ where $u$ is the unique solution to problem
\begin{equation}
\left\{
\begin{array}{rcll}
(-\Delta)^s u &= & |\nabla v|^{q}+\l f & \text{ in }\Omega , \\ u &=& 0 &\hbox{  in } \mathbb{R}^N\setminus\Omega,\\ u&>&0 &\hbox{ in }\Omega.
\end{array}%
\right.  \label{fix4}
\end{equation}%
By similar  computations as in the proof of Theorem \ref{fix00} and for  $\l<\l^*(f)$ fixed, we can prove that $T_q$ has a fixed point in $E_q$ and then  problem
\eqref{qqA} has a solution $u\in E_q$.
\end{proof}
\begin{remark} The result in Theorem \ref{fix011} is coherent with the local result $s=1$.

If $q$ is in the complementary interval, that is, $q>\dfrac{s}{1-s}$, we consider the modified problem
$$
\left\{
\begin{array}{rcll}
(-\Delta)^s u &= & d^{q(1-s)-s}(x)|\nabla u|^{q}+\l f & \text{ in }\Omega ,
\\ u &=& 0 &\hbox{  in } \mathbb{R}^N\setminus\Omega,\\
u&>&0 &\hbox{ in }\Omega.
\end{array}%
\right.
$$
Then the existence of a solution holds in the set
$$
E=\{v\in W^{1,1}_0(\O): v\, d^{1-s}\in W^{1,qm}_0(\O)\mbox{  and  } \bigg(\io |\n (v\, d^{1-s})|^{qm} dx\bigg)^{\frac{1}{qm}}\le l^{\frac{1}{q}}\}.
$$
\end{remark}
\begin{remark}
Notice that if $s$ and $m$ satisfy
\begin{equation*}\label{condiss-q}
\frac{N}{N+1}(1-\frac{1}{q'm})<s<1,
\end{equation*}
then existence of solution holds in the space
$$
E=\{v\in W^{1,1}_0(\O): v\in W^{1,qm}_0(\O)\mbox{  and  } \bigg(\io |\n v|^{qm} dx\bigg)^{\frac{1}{qm}}\le l^{\frac{1}{q}}\}.
$$
The proof follows for the same fixed point argument in the proof ot Theorem \ref{fix011}.
\end{remark}
\begin{remark}\label{rm112}
As in the case $q=2s$, since $m>\frac{N}{q'(2s-1)}$, by the same kind of arguments as in Remark \ref{rm11},  it holds that problem \eqref{qqA} has a unique
positive solution in the convex set $E_q$ that is the minimal solution of \eqref{qqA}.
\end{remark}
\section{A nonexistence result} In this section we prove the following partial result of nonexistence, that is a deep defference with the local case.
\begin{Theorem}\label{non1}
Assume that $f\in L^\infty(\O)$  is a nonnegative
function with $f\neq 0$. Then if $\a>\frac{1}{1-s}$, the problem
\begin{equation}\label{grad-nn}
\left\{
\begin{array}{rcll}
(-\Delta )^s u&=&|\nabla u|^{\alpha}+ f & \inn \Omega\\ 
u(x)&=&0 & \inn(\mathbb{R}^N\setminus\Omega),\\
\end{array}\right.
\end{equation}
has no weak solution $u\in W^{1,\a}(\Omega)$.
\end{Theorem}
\begin{proof}We argue by contradiction, assume a weak solution $v\in W^{1,\a}(\Omega)$ to problem \eqref{grad-nn} with $\alpha>\dfrac{1}{1-s}$.
By using the classical Hardy inequality we obtain that
$$\int_\Omega \frac{v^\alpha}{d^\alpha}dx\le \int_\Omega |\nabla v|^\alpha dx<+\infty.$$
By the results in \cite{RS} the solution behaves  as $v\backsimeq d^s$, therefore, as a consequence, 
$$\int_\Omega \dfrac{1}{d^{\alpha(1-s)}}dx<\infty.$$ 
then necessarily  $\alpha<\frac{1}{1-s}$, a contradiction with the hypothesis.
\end{proof}
\section{Some open problems}
The regularity result proved in Lemma \ref{estimmm} is the key in order to show the existence results and it is worthy point out that it depends directly on the
representation formula given in \eqref{repre} and in the pointwise estimates on the Green function $G_s$.

With the previous remark in mind we can formulate the following open problems that should be interesting to solve.
\begin{enumerate}
\item Let consider the operator $L_k$ defined by
$$
L_k(u)=P.V\int_{\ren} \,(u(x)-u(y))k(x,y)\,dy,
$$
where $k$ is a suitable symmetric function.  Consider the problem
\begin{equation*}
\left\{
\begin{array}{rcll}
L_k(u) &= & |\nabla u|^{q}+\l f & \text{ in }\Omega , \\ u &=& 0 &\hbox{  in } \mathbb{R}^N\setminus\Omega,\\ u&>&0 &\hbox{ in }\Omega,
\end{array}%
\right.
\end{equation*}
It seems to be interesting to find  conditions on $L_k$ in order to find the same kind of regularity results, for instance, getting estimates without the
explicit representation formula. This kind of results will be interesting in the pure fractional case to elucidate the existence of a solution if $\dfrac{1}{1-s}>q>\dfrac{s}{1-s}$.
\item In the local case $s=1$ and for the critical exponent $q=2$, an exponential regularity is obtained for any solution to problem \eqref{qqA}. See
    \cite{ADP}.  Precisely the result is that any positive solution satisfies $e^{\a u}-1\in \sob(\O)$ for all $\a<\frac 12$. It seems to be natural to ask
    for the optimal regularity in the fractional case. \item Consider the nonlinear operator
$$ (-\D^s_{p})\, u(x):=P.V\int_{\ren} \,\dfrac{|u(x)-u(y)|^{p-2}(u(x)-u(y))}{|x-y|^{N+ps}} \,dy$$
with $1<p<N$ and  $s\in (0,1)$, then as it was proved in \cite{AAB}, the problem
\begin{equation*}
\left\{
\begin{array}{rcll}
(-\D^s_{p}) u & = & f(x)  & \text{ in } \O, \\ u &= & 0 & \text{ in }\ren\setminus\O,
\end{array}%
\right.
\end{equation*}
has a unique entropy solution for nonnegative datum. It seems to be interesting to show the regularity of $|\n u|$ if  $sp'>\frac{(2-p)N}{p-1}+1$ and to
consider the nonlinear nonlocal version of problem \eqref{P}. \item A problem with nonlocal diffusion and nonlocal  growth term could be formulated for all
$s\in (0,1)$ and it should be interesting to analyze it in detail.
\end{enumerate}

\end{document}